\documentclass[a4paper,reqno]{amsart}
\usepackage[paperheight=11in,paperwidth=8.5in]{geometry}
\usepackage[utf8]{inputenc}
\usepackage{amsmath,amsfonts,amsthm,amssymb}
\usepackage{graphicx,enumerate}
\usepackage{mathtools,bm}

\ifdefined\pdfoutput\pdfoutput=1\fi
\makeatletter
\def\l@section{\@tocline{1}{12pt plus2pt}{0pt}{}{\bfseries}}
\def\l@subsection{\@tocline{2}{0pt}{2pc}{2pc}{}}
\makeatother
\makeatletter
\def\subsection{\@startsection{subsection}{2}{\z@}%
	{-3.25ex\@plus -1ex \@minus -.2ex}%
	{1.5ex \@plus .2ex}%
	{\normalfont\sffamily\itshape}}
\def\subsubsection{\@startsection{subsubsection}{3}%
	\z@{.5\linespacing\@plus.7\linespacing}{-.5em}%
	{\normalfont\itshape}}
\renewcommand\paragraph{\@startsection{paragraph}{4}{\z@}%
	{3.25ex \@plus1ex \@minus.2ex}%
	{-1em}%
	{\normalfont\normalsize\bfseries}}
\makeatother

\usepackage[dvipsnames]{xcolor}
\definecolor{RawSeinna}{named}{RawSienna}
\definecolor{Sepia}{named}{Sepia}
\usepackage[colorlinks=true, pdfstartview=FitV, linkcolor=Brown, citecolor=Brown, urlcolor=Brown,pdfencoding=auto]{hyperref}

\numberwithin{equation}{section}
\newtheorem{theorem}{Theorem}[section]
\newtheorem*{theorem*}{Theorem}

\newtheorem{lemma}[theorem]{Lemma}
\newtheorem{proposition}[theorem]{Proposition}
\newtheorem{conjecture}[theorem]{Conjecture}

\theoremstyle{definition}

\newcommand{\R}{\mathbb{R}}

\newcommand{\N}{\mathbb{N}}

\def\beq{\begin{equation}}
\def\eeq{\end{equation}}

\def\D{\mathcal{D}}

\newcommand{\nD}{{\Delta\kern-0.5em/}}

\DeclareMathOperator{\supp}{supp}
\DeclareMathOperator{\sgn}{sgn}

\allowdisplaybreaks
\title[]{\textbf{Bilinear Operators Along Variable Curves}}

\author[N. Gismondi]{Nicholas Gismondi}
\address[N. Gismondi]{Department of Mathematics, Purdue University, IN 47907, USA}
\email{ngismond@purdue.edu}

\author[M. Pathak]{Mandon Pathak}
\address[M. Pathak]{Department of Mathematics, Purdue University, IN 47907, USA}
\email{pathak30@purdue.edu}


\date{\today}

\begin{document}

\begin{abstract}
Motivated by Zygmund's differentiation conjecture and Stein's conjecture for Lipschitz vector fields, we study bilinear singular integral and maximal operators along variable curves given by
$$
\begin{aligned}
BH_a^\alpha(f,g)(x) &= \operatorname{p.v.}\int_{\mathbb R}
f(x-t) g\left(x + a(x)t^\alpha\right)\,\frac{dt}{t},\\
BM_a^\alpha(f,g)(x) &= \sup_{r>0}\frac{1}{2r}\int_{-r}^{r}
\left|f(x-t)g\left(x+a(x)t^\alpha\right)\right|\,dt,\\
BSH_c^\alpha(f,g)(x) &= \operatorname{p.v.} \int_{\mathbb R} f(x-t) g(x + t^\alpha)\, \frac{dt}{t-c(x)},
\end{aligned}
$$
where $\alpha > 0$ and $a,c:\mathbb R\to\mathbb R$ are measurable.
For every $\alpha>0$ and $0<A<B<\infty$, we construct a smooth, globally Lipschitz coefficient $a:\mathbb R\to(A,B)$ such that, for every H\"{o}lder conjugate pair $1<p,q<\infty$, no strong--type $L^p(\mathbb R)\times L^q(\mathbb R)\to L^1(\mathbb R)$ estimate holds for $BH_a^\alpha$ or $BM_a^\alpha$ on $C_c^\infty(\mathbb R)\times C_c^\infty(\mathbb R)$; moreover, there exist nonnegative smooth inputs in $\bigcap_{s>1}L^s(\mathbb R)$, one of which is compactly supported, whose outputs under $BH_a^\alpha$ and $BM_a^\alpha$ do not belong to $L^1(\mathbb R)$. For any measurable $a:\mathbb R\to\mathbb R$, $BM_a^\alpha$ maps $L^p(\mathbb R) \times L^q(\mathbb R)$ boundedly into $L^r(\mathbb R)$ whenever $1<p,q\leq\infty$, $1/r = 1/p + 1/q$, and $r > 1$, and maps $L^p(\mathbb R) \times L^q(\mathbb R)$ into weak $L^1(\mathbb R)$ when $1/p+1/q=1$, with constants independent of $a$. Finally, we investigate the shifted operator $BSH_c^\alpha$, including dyadic-integer-valued $c$, and prove that there is a measurable choice of $c$ such that the operator admits no bounded $L^p(\mathbb R) \times L^q(\mathbb R) \to L^{r,\infty}(\mathbb R)$ estimates in the range $1 \leq p,q \leq \infty$ and $\frac{1}{2}<r<\infty$, resolving an open problem posed by Gaitan and Lie.
\end{abstract}

\maketitle

\section{Introduction}
\subsection{Motivation and Background}
The study of averages along spatially varying curves is closely connected with Zygmund's differentiation conjecture. In its $L^2$ formulation, this conjecture asserts that, for a Lipschitz unit vector field $v:\mathbb R^2\to\mathbb S^1$ and every $F\in L^2(\mathbb R^2)$,
$$
\lim_{r \to 0^+} \frac{1}{2r}\int_{-r}^{r} F(z-tv(z))\,dt = F(z)
\qquad \text{for almost every }z\in\mathbb R^2.
$$
The associated local maximal operator is
$$
M_{v,r_0}F(z) := \sup_{0<r<r_0} \frac{1}{2r}\int_{-r}^{r}|F(z-tv(z))|\, dt,
$$
where $r_0 > 0$ is chosen sufficiently small relative to the Lipschitz constant of $v$. A weak--type $(2,2)$ estimate for this operator would imply the differentiation statement by approximation; see \cite{UnifiedLie} for background. One can of course consider a singular integral analogue given by
$$
H_{v,r_0}F(z) = \int_{-r_0}^{r_0} F(z-tv(z))\, \frac{dt}{t} \ \footnote{All integrals which contain a singularity are understood in the principal value sense.}.
$$
The Stein conjecture asserts that $H_{v,r_0}$ is bounded from $L^2(\R^2)$ into $L^2(\R^2)$ for $v$ Lipschitz and $r_0$ chosen small enough depending on the Lipschitz norm of $v$. Until extremely recently, the best result in this direction was the work of Lacey and Li, who showed that if the so-called \textit{Lipschitz Kakeya maximal function} is $L^p(\R^2)$ to $L^p(\R^2)$ bounded for some $1 < p < 2$, then the Stein conjecture for $v \in C^{1+\eta}(\R^2)$ holds \cite{LaceyLi}. However, the full Stein conjecture was in fact shown to hold for $v$ Lipschitz in the recent groundbreaking work from OpenAI \cite{SteinConjecture}. The proof establishes a strong $L^2(\R^2)$ bound for $H_{v,r_0}$ uniformly over all inner truncations at outer scale $r_0 \lesssim \operatorname{Lip}(v)^{-1}$, answering Stein's weak-$(2,2)$ conjecture in the affirmative. Using spatial localization and a Lipschitz commutator, Littlewood--Paley reassembly reduces the global bound to a logarithmically improved restricted estimate on a single vertical frequency band, which is decomposed into narrow vertical bins and a Hilbert-valued ``popular forest'' of wave packets. While Lacey and Li's weak-$(2,2)$ Lipschitz Kakeya estimate and L\'{e}pingle's martingale $3$--variation inequality \cite{Lepingle, ZorinKranich} control the popular top count and the number of separated significant directions at each angular scale up to polynomial losses in the set-measure disparity $L = \log\sqrt{|E|/|F|}$, the necessary oscillatory saving comes from proving a sharp finite Fourier inequality with exact operator norm $1$ and an improved coefficient exponent, iterated across a multiscale angular hierarchy---thereby bypassing both the open $L^p(\R^2)$, $1 < p < 2$, Lipschitz Kakeya conjecture and the $C^{1+\eta}(\R^2)$ regularity assumption of \cite{LaceyLi}.\footnote{See the AI Declaration in Section \ref{sec:AI_dec}.}

Lie considered the curved variants of each of these operators in \cite{UnifiedLie} and showed that they are $L^p(\R^2)$ bounded for $1 < p < \infty$. The reason for this is that the curvature present breaks the modulation invariance in the operator, which allows more classical techniques like $TT^*$, van der Corput, and Littlewood--Paley estimates to be employed, as opposed to the heavy machinery of wave packet analysis that is required to attack the ``flat'' versions of the operators. However, as we will demonstrate below, estimates for the linear and sublinear operators on $\R^2$ do not imply bounds for the bilinear and sub-bilinear variants that we consider. One of our primary goals is to determine to what extent these bounds carry over to this setting.

For $f,g \in \mathcal{S}(\R)$, we consider the following bilinear integral operators:
\begin{equation}\label{eq:bilin_Hilbert}
    BH_a^\alpha(f,g)(x) = \int_{\R} f(x-t) g\left(x+a(x)t^\alpha\right)\, \frac{dt}{t},
\end{equation}
\begin{equation}\label{eq:bilin_max}
    BM_a^\alpha(f,g)(x) = \sup_{r > 0} \frac{1}{2r} \int_{-r}^r \left|f(x-t) g\left(x + a(x)t^\alpha\right)\right|\, dt,
\end{equation}
and
\begin{equation}\label{eq:bilin_shift}
    BSH^\alpha_c(f,g)(x) = \int_{\R} f(x-t) g(x+t^\alpha)\, \frac{dt}{t - c(x)},
\end{equation}
where $a,c:\R \to \R$ are assumed to be measurable and $\alpha > 0$\footnote{For integral $\alpha$, the usual real power is used. For nonintegral $\alpha$ and $t < 0$, either of the conventions $t^\alpha = \sgn(t) |t|^\alpha$ or $t^\alpha = |t|^\alpha$ on the negative half-line is allowed; all principal lower bounds below use only $t>0$, and possible contributions from $t<0$ are checked separately under both conventions.}. The study of the operators proposed in \eqref{eq:bilin_Hilbert}--\eqref{eq:bilin_shift} originates in the work of Lie \cite{UnifiedLie} and later Gaitan and Lie \cite{GaitanLie2}.

The three operators in \eqref{eq:bilin_Hilbert}--\eqref{eq:bilin_shift} correspond to three different pointwise selection mechanisms: selection of the curve coefficient, selection of the spatial averaging scale, and selection of the location of the kernel singularity. Although their definitions are closely related, each operator isolates a distinct analytic difficulty.

The operator $BH_a^\alpha$ belongs to the theory of bilinear singular integrals along curves. If $a(x)\equiv b$ is constant and $\alpha = 1$, then
\begin{equation*}\label{eq:BHT-special-case}
    BHT_b(f,g)(x) := BH_b^1(f,g)(x) = \int_{\mathbb{R}} f(x-t)g(x+bt)\,\frac{dt}{t},
\end{equation*}
which is the classical bilinear Hilbert transform with slope $b$ (nondegenerate whenever $b \notin \{0,-1\}$). Originally it arose in connection with the Cauchy integral along Lipschitz curves, where Calder\'{o}n suggested that proving boundedness of $BHT_b$ might provide an avenue to establish $L^2$ boundedness of the Cauchy transform along Lipschitz curves \cite{Calderon, MuscaluSchlag}. Ultimately, the Cauchy transform along Lipschitz curves was proven to be $L^2$ bounded by Coifman, McIntosh, and Meyer using other means \cite{CMM}, while the boundedness of the bilinear Hilbert transform remained open for nearly two decades. In one of the seminal achievements of time--frequency analysis, Lacey and Thiele established $L^p$ boundedness of $BHT_b$ using wave packet analysis \cite{LaceyThiele1, LaceyThiele2}. Further developments by Grafakos and Li \cite{GL} and later Li \cite{UniformLi} showed that the bounds are uniform in $b$ across certain $L^p$ ranges.

When $a(x)\equiv b$ and $\alpha\neq1$, the operator becomes the bilinear Hilbert transform along the monomial curve
\begin{equation*}
    \Gamma_b^\alpha(t) = \left(t,-bt^\alpha\right).
\end{equation*}
For $b \neq 0$ and $t \neq 0$, this curve has nonvanishing curvature whenever $\alpha \neq 1$. Thus the parameter $\alpha$ distinguishes the zero-curvature regime $\alpha = 1$ from the genuinely curved regime $\alpha \neq 1$. The latter combines singular-integral cancellation with oscillatory decay generated by curvature. Li was the first to consider the boundedness of this curved bilinear Hilbert transform, establishing $L^2 \times L^2 \to L^1$ boundedness for $b = 1$ and integer $\alpha \geq 2$ \cite{Li}. Lie subsequently extended this result to a broad class of non-flat smooth curves, proving $L^p \times L^q \to L^r$ boundedness for $1 < p < \infty$, $1 < q \leq \infty$, and $1 \leq r < \infty$ with $1/p + 1/q = 1/r$ \cite{BHCurvatureLie1, BHCurvatureLie2}.

Allowing $a(x)$ to depend measurably on $x$ introduces a further difficulty. At each output point $x$, the operator selects a possibly different curve from the family
\begin{equation*}
    \Gamma_u^\alpha(t) = \left(t,-u\,t^{\alpha}\right),
    \qquad u\in\mathbb{R}.
\end{equation*}
Consequently, the operator is no longer translation invariant in the output variable, and its underlying time--frequency geometry changes from point to point.

This formulation may also be understood as a linearization of the coefficient-maximal operator
\begin{equation}\label{eq:coefficient-maximal}
    BH_*^\alpha(f,g)(x) := \sup_{u\in\mathbb{R}} \left|\int_{\mathbb{R}} f(x-t) g\left(x+u\,t^{\alpha}\right)\,\frac{dt}{t}\right|.
\end{equation}
Indeed, a measurable choice $u=a(x)$ approximately realizing the supremum in \eqref{eq:coefficient-maximal} reduces the maximal problem to estimates for $BH_a^\alpha$ that are uniform over all admissible measurable functions $a$. The measurable coefficient in \eqref{eq:bilin_Hilbert} therefore records the pointwise selection inherent in the corresponding maximal problem.

Regarding the history of bilinear maximal operators, Lacey first showed that the ``flat'' bilinear maximal function
$$
BM(f,g)(x) := BM_1^1(f,g)(x) = \sup_{r > 0} \frac{1}{2r} \int_{-r}^r |f(x-t) g(x+t)|\, dt
$$
maps $L^p(\R) \times L^q(\R)$ boundedly into $L^r(\R)$ for $r > 2/3$ \cite{LaceyBilinMax}. The curved variant ($\alpha \neq 1$, $a(x) \equiv 1$) was subsequently investigated by Gaitan and Lie, who proved that it maps boundedly into $L^r(\R)$ for $r \geq 1$, and moreover that this range of Lebesgue exponents is sharp \cite{GaitanLie1}. It is somewhat surprising that introducing curvature narrows the range of admissible target exponents $r$, eliminating bounded mappings into quasi-Banach spaces $L^r(\R)$ with $r < 1$.

The connection with the planar theory can be made explicit
through tensor products. Define
$$
\begin{aligned}
H_a^\alpha F(x,y) &:= \int_{\mathbb R}
F\left(x-t,y+a(x)t^\alpha\right)\,\frac{dt}{t},\\
M_a^\alpha F(x,y) &:= \sup_{r>0}\frac{1}{2r}\int_{-r}^{r}
\left|F\left(x-t,y+a(x)t^\alpha\right)\right|\,dt.
\end{aligned}
$$
For $F = f\otimes g$, where $F(x,y) = f(x)g(y)$, we have
the exact identities
\begin{equation*}\label{eq:diagonal-restriction}
\begin{aligned}
BH_a^\alpha(f,g)(x) &= H_a^\alpha(f\otimes g)(x,x),\\
BM_a^\alpha(f,g)(x) &= M_a^\alpha(f\otimes g)(x,x).
\end{aligned}
\end{equation*}
When $\alpha = 1$, the planar operators act along lines with direction $(1,-a(x))$, placing them within the framework of directional operators whose direction depends on one coordinate. When $\alpha \ne 1$ and $a(x) \ne 0$, the corresponding paths have nonvanishing curvature away from $t=0$.

Thus the bilinear operators are diagonal restrictions of
planar operators along variable curves. This restriction
introduces an additional analytic difficulty: ordinary
$L^p(\mathbb R^2)$ control does not provide a bounded trace
on the diagonal. Consequently, planar boundedness results
do not automatically yield the corresponding bilinear
estimates. \textit{The identities above provide a concrete link with the variable-curve theory surrounding Zygmund's and Stein's conjectures, while identifying the additional restriction problem inherent in the bilinear setting.}

The maximal operator $BM_a^\alpha$ retains the same curved geometry as $BH_a^\alpha$, while removing the cancellation of the odd kernel $dt/t$ and replacing it with positive averages. It measures the concentration of
\begin{equation*}
    f(x-t)g\left(x+a(x)t^\alpha\right)
\end{equation*}
along the relevant family of curves, while the supremum over $r>0$ introduces a pointwise selection over spatial scales. Studying $BH_a^\alpha$ and $BM_a^\alpha$ together therefore allows one to separate those estimates that follow from geometric size considerations from those that depend essentially on oscillation and cancellation. Despite this difference, the two operators share the same basic scale decomposition and much of the same stationary and nonstationary geometry; for a detailed comparison of their analysis, see \cite{GaitanLie2}.

The shifted operator $BSH_{c}^{\alpha}$ exhibits a different phenomenon. Here the curve is fixed, while the singularity of the kernel is allowed to vary with $x$. After making the change of variables $t = s + c(x)$, we obtain
\begin{equation}\label{eq:1.4}
BSH_{c}^{\alpha}(f,g)(x) =  \int_{\mathbb{R}}f\left(x-c(x)-s\right)g\left(x+(s+c(x))^{\alpha}\right)\frac{ds}{s}.
\end{equation}
Equation \eqref{eq:1.4} shows that translating the kernel singularity is generally not a harmless modification: the translation simultaneously shifts the spatial argument of the first input and deforms the curved argument of the second input.

In the linear case $\alpha=1$, a constant shift $c$ can be absorbed into opposite translations of the two inputs. Indeed, writing
\begin{equation*}
    \tau_u f(y):=f(y-u),
\end{equation*}
we have
\begin{equation*}\label{eq:linear-shifted-BHT}
    BSH_c^1(f,g)
    =
    BHT_1\left(\tau_c f,\tau_{-c}g\right),
    \qquad c\in\mathbb{R}.
\end{equation*}
For a nonlinear curve, however, the quantity
$(s+c)^\alpha$ cannot generally be separated into a pure power of $s$ and a constant translation. For example, in the cubic case $\alpha = 3$,
\begin{equation*}\label{eq:cubic-shift}
    (s+c)^3 = s^3+3cs^2+3c^2s+c^3.
\end{equation*}
Therefore, translating the singularity of the cubic operator generates quadratic and linear terms. When $c=c(x)$, the coefficients of these lower-order terms become measurable functions of $x$. The operator
$BSH_c^3$ consequently combines curvature, variable coefficients, and translation of the kernel singularity within a single bilinear model.

The shifted formulation is naturally associated with the maximal operator
\begin{equation}\label{eq:1.5}
BSH_{\mathcal{C}}^{\alpha}(f,g)(x) := \sup_{c\in\mathcal{C}}\left| \int_{\mathbb{R}}f(x-t)g(x+t^{\alpha})\frac{dt}{t-c}\right|, 
\end{equation}
where $\mathcal{C} \subseteq \mathbb{R}$. A measurable selection $c=c(x)$ linearizes the supremum in \eqref{eq:1.5}. In contrast with Carleson--type operators, where the maximal parameter acts through modulation, the parameter $c$ in \eqref{eq:1.5} acts by translating the singular kernel. The study of $BSH_{c}^{\alpha}$ therefore leads to a translation-maximal problem in a bilinear curved setting.

The boundedness problem for the shifted maximal operator was formulated in \cite[Open Problem~2]{GaitanLie2}. Its parabolic specialization also connects directly with Open Problem~1 of the same work, which asks whether the boundedness results for the bilinear
Hilbert transform and maximal operator along $at+bt^\alpha$ extend to the quadratic case $\alpha=2$, for fixed $a \in \mathbb{R} \setminus \{-1\}$ and $b \in \mathbb{R}$, in the range
$$
\frac{1}{p} + \frac{1}{q} = \frac{1}{r},
\qquad p,q>1,
\qquad r>\frac{2}{3}.
$$
For the singular integral, this connection is expressed by a translation identity. Writing $\tau_s h(x):=h(x-s)$ and making the substitution $u=t-1$, we obtain
\begin{align*}
BSH_1^2(\tau_{-1}f,\tau_1g)(x)
&= \int_{\mathbb R}
f(x-t+1)\,g(x+t^2-1)\,\frac{dt}{t-1}\\
&= \int_{\mathbb R}
f(x-u)\,g(x+2u+u^2)\,\frac{du}{u}\\
&=H_{2,1}^2(f,g)(x).
\end{align*}
Since translations preserve Lebesgue norms, the fixed-shift operator
$BSH_1^2$ and the quadratic hybrid operator $H_{2,1}^2$ have identical $L^p\times L^q\to L^r$ boundedness properties. This motivates the study of $BSH_c^\alpha$: if one were able to show that $BSH_c^2$ was bounded, this would imply boundedness of the hybrid operator with quadratic resonances considered in \cite{GaitanLie2}.

Our contribution is a complete resolution of the Lebesgue-space boundedness of $BM_a^\alpha$ and $BSH_c^\alpha$, together with sharp $L^1$-endpoint counterexamples for both $BH_a^\alpha$ and $BM_a^\alpha$ having strictly positive, bounded, globally Lipschitz coefficients.

For $BH_a^\alpha$ and $BM_a^\alpha$ at the $r=1$ endpoint, a natural question is whether failure of $L^p \times L^q \to L^1$ boundedness can only be triggered by pathological behavior of the coefficient $a(x)$---such as $a(x)$ becoming unbounded, $a(x)$ approaching $0$ where the curvature degenerates, or $a(x)$ taking the resonant slope $-1$ when $\alpha = 1$. We show that none of these degeneracies is needed. For any $\alpha > 0$ and any prescribed positive interval $(A,B) \subset (0,\infty)$, we construct a single smooth, globally Lipschitz coefficient $a : \R \to (A,B)$ for which both $BH_a^\alpha$ and $BM_a^\alpha$ fail to satisfy any $L^p(\R) \times L^q(\R) \to L^1(\R)$ norm estimate on $C_c^\infty(\R) \times C_c^\infty(\R)$ for all $1 < p, q < \infty$ with $1/p + 1/q = 1$. Moreover, using an infinite train of disjoint smooth bumps, we construct nonnegative smooth inputs in $\bigcap_{s>1} L^s(\R)$, one of which lies in $C_c^\infty(\R)$, for which neither $BH_a^\alpha(f,g)$ nor $BM_a^\alpha(f,g)$ belongs to $L^1(\R)$. This $L^1$-endpoint failure may be compared with the directional Hilbert transform counterexample of Karagulyan \cite{Karagulyan}, in that both establish the failure of strong--type estimates for Hilbert transform--esque operators when the target exponent matches the dimension of the underlying target space. Boundedness of $BH_a^\alpha$ for $r > 1$ is left as an open problem in Conjecture \ref{open_problem}.

On the positive side, for $BM_a^\alpha$ with an arbitrary measurable coefficient $a:\R\to\R$, we prove strong--type $L^p(\R) \times L^q(\R) \to L^r(\R)$ boundedness for all $r > 1$, as well as weak--type $L^p(\R) \times L^q(\R) \to L^{1,\infty}(\R)$ boundedness at the endpoint $r = 1$, with operator norm bounds independent of $a$. The proof is elementary and rests on the pointwise estimate
$$
BM_a^\alpha(f,g)(x) \lesssim_{\alpha,s} \left(M(|f|^s)(x)\right)^{1/s} \left(M(|g|^{s'})(x)\right)^{1/s'}
$$
for H\"{o}lder conjugate exponents $s, s' > 1$. While the flat case ($\alpha = 1$, $a \equiv 1$) of this pointwise bound is classical (see e.g.\ \cite{LaceyBilinMax}), our modest contribution is to show that it persists for every $\alpha > 0$ and any measurable $a(x)$, with an implicit constant depending only on $\alpha$ and $s$. Combining our positive and negative results for $BM_a^\alpha$ reveals a somewhat surprising progression in the sharp strong--type target range for bilinear maximal averages: the flat constant-coefficient operator maps into $L^r(\R)$ for $r > 2/3$ \cite{LaceyBilinMax}, fixed non-flat monomial curvature restricts the strong--type range to $r \geq 1$ \cite{GaitanLie1}, and allowing a spatially varying Lipschitz coefficient $a(x) \in (A,B)$ removes the strong--type $L^1$ endpoint as well, leaving $r > 1$ (and weak $L^1$ at $r=1$) as the sharp range.

For $BSH_c^\alpha$, we prove that there exists a measurable function $c$ (even taking values in the dyadic integers $\{2^j : j \in \N\}$, or taking values in $[0,1]$) such that no strong or weak--type $L^p(\R) \times L^q(\R) \to L^r(\R)$ estimates hold for $1/2 < r < \infty$, resolving Open Problem~2 of Gaitan and Lie \cite{GaitanLie2}. This demonstrates that attacking the hybrid operator with quadratic resonances requires taking structural regularity of the shift into account rather than seeking uniform measurable-shift estimates.

We state our results formally in the following subsection.

\subsection{Main Results}

\begin{theorem}\label{thm:main1}
Fix $\alpha > 0$ and $0 < A < B < \infty$. There is a smooth, globally Lipschitz function $a : \R \to (A,B)$ such that, for every $1 < p, q < \infty$ with $1/p + 1/q = 1$, no finite constant $C = C(\alpha,p,q)$ satisfies
\begin{equation}\label{eq:BH-norm-failure}
\Vert BH_a^\alpha(f,g) \Vert_{L^1(\R)} \leq C \Vert f \Vert_{L^p(\R)} \Vert g \Vert_{L^q(\R)}, \qquad f, g \in C_c^\infty(\R).
\end{equation}
The coefficient $a$ is independent of the number of test bumps and of $p, q$. Moreover, there exist nonnegative smooth inputs $f, g \in \bigcap_{s > 1} L^s(\R)$, one of which is compactly supported, for which
$$
BH_a^\alpha(f,g) \not\in L^1(\R).
$$
\end{theorem}

\begin{theorem}\label{thm:main2}
For all $\alpha > 0$, any measurable function $a: \R \to \R$, and any exponents satisfying
$$
1 < p,q \leq \infty, \qquad r > 1, \qquad \frac{1}{r} = \frac{1}{p} + \frac{1}{q},
$$
the operator $BM_a^\alpha$ is bounded from $L^p(\R) \times L^q(\R)$ into $L^r(\R)$: there is a constant $C = C(\alpha,p,q)$, independent of $a$, such that
$$
\Vert BM_a^\alpha(f,g) \Vert_{L^r(\R)} \leq C \Vert f \Vert_{L^p(\R)} \Vert g \Vert_{L^q(\R)}.
$$
Furthermore, for any H\"{o}lder conjugate pair $1 \leq p, q \leq \infty$ with $1/p + 1/q = 1$, $BM_a^\alpha$ is bounded from $L^p(\R) \times L^q(\R)$ into $L^{1,\infty}(\R)$, again with a constant $C = C(\alpha,p,q)$ independent of $a$.

On the other hand, at the strong--type $r=1$ endpoint, fix $\alpha > 0$ and $0 < A < B < \infty$. For the same smooth, globally Lipschitz function $a : \R \to (A,B)$ from Theorem \ref{thm:main1} and every $1 \leq p, q \leq \infty$ with $1/p + 1/q = 1$, no finite constant $C = C(\alpha,p,q)$ satisfies
\begin{equation}\label{eq:BM-norm-failure}
\Vert BM_a^\alpha(f,g) \Vert_{L^1(\R)} \leq C \Vert f \Vert_{L^p(\R)} \Vert g \Vert_{L^q(\R)}, \qquad f, g \in C_c^\infty(\R).
\end{equation}
Moreover, there exist nonnegative smooth inputs $f, g \in \bigcap_{s > 1} L^s(\R)$, one of which is compactly supported, for which
$$
BM_a^\alpha(f,g) \not\in L^1(\R).
$$
\end{theorem}

\begin{theorem}\label{thm:main3}
For all $\alpha > 0$ there is a choice of measurable function $c:\R \to \R$ such that no finite constant $C = C(c,\alpha,p,q)$ exists for which
$$
\Vert BSH_c^\alpha(f,g) \Vert_{L^r(\R)} \leq C \Vert f \Vert_{L^p(\R)} \Vert g \Vert_{L^q(\R)}, \qquad \frac{1}{r} = \frac{1}{p} + \frac{1}{q}, \qquad \frac{1}{2} < r < \infty, \quad 1 < p,q \leq \infty
$$
holds for all $f,g \in \mathcal{S}(\R)$.
\end{theorem}

We may deduce Theorem \ref{thm:main3} from either of the following two propositions.

\begin{proposition}\label{prop:unit-interval-shift}
Let $\alpha>0$. Suppose that
$$
1<p,q\leq\infty,
\qquad
\frac12<r<\infty,
\qquad
\frac1p+\frac1q=\frac1r.
$$
Then the maximal shifted operator
$$
BSH_{[0,1]}^\alpha(f,g)(x)
:=
\sup_{c\in[0,1]}
\left| \int_{\mathbb R}
f(x-t)g(x+t^\alpha)\,\frac{dt}{t-c}
\right|
$$
does not satisfy an estimate of the form
\begin{equation}\label{eq:unit-interval-weak-bound}
\|BSH_{[0,1]}^\alpha(f,g)\|_{L^{r,\infty}(\mathbb R)}
\leq
C
\|f\|_{L^p(\mathbb R)}
\|g\|_{L^q(\mathbb R)}
\end{equation}
with a constant $C$ independent of $f$ and $g$. In particular,
$$
BSH_{[0,1]}^\alpha:
L^p(\mathbb R)\times L^q(\mathbb R)
\textit{ is not bounded into }
L^r(\mathbb R).
$$
\end{proposition}

\begin{proposition}\label{prop:dyadic}
Let $\mathcal D=\{2^j:j\in\mathbb N\}$. For every $\alpha>0$ there
exists a measurable function $c:\mathbb R\to\mathcal D$\footnote{This
dyadic model was suggested by Lie to the second author in personal
communication as a possible setting for bilinear boundedness estimates.}
such that, for every $1\leq p,q\leq\infty$ with $1/r=1/p+1/q$, the
estimate
$$
 \|BSH_c^\alpha(f,g)\|_{L^r(\mathbb R)}
 \leq C\|f\|_{L^p(\mathbb R)}\|g\|_{L^q(\mathbb R)}
$$
fails for every finite $C = C(c,\alpha,p,q)$. If $r < \infty$, the
estimate fails even with $L^r(\R)$ replaced by $L^{r,\infty}(\R)$.
All counterexamples use inputs in $C_c^\infty(\mathbb R)$, and the
same measurable function $c$ works for every exponent pair.
\end{proposition}

\subsection{Notation}
Standard background on Lebesgue and weak $L^p$ spaces can be found in \cite{Grafakos1} and \cite{Grafakos2}. For $0 < p < \infty$, $L^p(\R)$ denotes the classical Lebesgue space
$$
L^p(\R) = \left\{f : \Vert f \Vert_{L^p(\R)} := \left(\int_{\R} |f(x)|^p\, dx\right)^{1/p} < \infty\right\},
$$
with the standard essential supremum norm when $p = \infty$. For $p \geq 1$ this is a Banach space, whereas for $0 < p < 1$ it is a quasi-Banach space. The weak Lebesgue space $L^{p,\infty}(\R)$ consists of measurable functions $f$ satisfying
$$
\Vert f \Vert_{L^{p,\infty}(\R)} := \sup_{\lambda > 0} \lambda \,|\{x \in \R : |f(x)| > \lambda\}|^{1/p} < \infty.
$$

As is standard in harmonic analysis, we write $A\lesssim B$ whenever there is a constant $C>0$ depending only on fixed parameters such that $A \leq CB$, and $A\simeq B$ when $A\lesssim B$ and $B\lesssim A$ hold simultaneously. Subscripts such as $A \lesssim_\alpha B$ indicate that the implicit constant depends on the displayed parameters.

\subsection{Acknowledgments}N.G. was supported by the NSF through grant DMS-2400238. The second author is grateful to his advisor, Victor Lie, for insightful discussions and suggestions that helped refine the formulation of the conjecture \eqref{open_problem2}. 

\section{Proof of Theorem \ref{thm:main1}}\label{sec:proof-thm1}

We prove Theorem \ref{thm:main1} by dividing the construction into three regimes according to whether $\alpha = 1$, $0 < \alpha < 1$, or $\alpha > 1$. Throughout all three cases we set the baseline value
$$
a_* = \frac{1}{2}(A+B).
$$
In each regime we first construct a single smooth, globally Lipschitz function $a : \R \to (A,B)$ and a family of compactly supported smooth test functions demonstrating the failure of \eqref{eq:BH-norm-failure} for all $1 < p, q < \infty$ with $1/p + 1/q = 1$. We then replace the finite bump train by an infinite bump train in Subsection \ref{subsec:infinite-train} to obtain a single pair of smooth inputs in $\bigcap_{s>1}L^s(\R)$ (one compactly supported) whose output fails to lie in $L^1(\R)$.

\subsection{The linear case: \texorpdfstring{$\alpha = 1$}{alpha = 1}}\label{subsec:alpha-eq-1}
Put $a_* = \frac{1}{2}(A+B)$ and choose numbers
$$
A < A_0 < a_- < a_+ < B_0 < B.
$$
Take a rapidly increasing geometric sequence $R_j = R Q^j$, $j \geq 1$, with $R, Q$ sufficiently large. Define the intervals
$$
I_j = \left[\frac{R_j}{1+a_+}, \frac{R_j}{1+a_-}\right], \qquad J_j = \left(\frac{R_j}{1+B_0}, \frac{R_j}{1+A_0}\right).
$$
Choose $Q > (1+B_0)/(1+A_0)$, so that the intervals $J_j$ are pairwise disjoint, and choose $R$ large enough that all of them lie to the right of $2$. On $J_j$, set
$$
a_j(x) = \frac{R_j}{x} - 1.
$$
Then $a_j(J_j) \subset (A_0, B_0)$. Choose smooth cutoffs $0 \leq \chi_j \leq 1$, supported in $J_j$ and equal to $1$ on $I_j$, obtained by dilating one fixed reference cutoff by $R_j$. Define
\begin{equation}\label{eq:a-def-linear}
a(x) = a_* + \sum_{j \geq 1} \chi_j(x)\left(a_j(x) - a_*\right).
\end{equation}
Because the intervals $J_j$ are disjoint and escape to infinity, the sum is locally finite, at most one term is nonzero at each point, and $a(x) \in (A,B)$ by convexity. On $J_j$, both $a_j'(x) = -R_j/x^2$ and the cutoff derivative contribution $\chi_j'(x)(a_j(x)-a_*)$ are $O(R_j^{-1})$. Thus $a$ is smooth and globally Lipschitz on $\R$.

Fix a nonzero function $f \in C_c^\infty((0,1))$ with $f \geq 0$, and put $m_f = \int_0^1 f(u)\,du > 0$. Choose a nonnegative function $\psi \in C_c^\infty(\R)$ with $\psi = 1$ on $[-B, 0]$. Increase $R, Q$ if necessary so that the supports of $\psi(\cdot - R_j)$ are pairwise disjoint. For each integer $N \geq 1$, define the finite bump train
\begin{equation}\label{eq:gN-linear}
g_N(y) = \sum_{j=1}^N \psi(y - R_j).
\end{equation}
For $x \in I_j$, $1 \leq j \leq N$, and $u \in \supp f \subset (0,1)$, the substitution $t = x - u$ gives
$$
t = x - u > 0, \qquad x + a(x)(x - u) = R_j - a(x)u \in [R_j - B, R_j],
$$
since $a(x) = a_j(x) \in (0, B)$ on $I_j$. Therefore the $j$th bump $\psi(\cdot - R_j)$ equals $1$ throughout the domain of integration. Moreover, every contribution to the integral from the remaining bumps is nonnegative because $f, \psi \geq 0$ and $f(x-t) \neq 0$ forces $t = x - u > 0$. Hence
\begin{equation}\label{eq:BH-lower-linear}
BH_a^1(f, g_N)(x) \geq \int_0^1 \frac{f(u)}{x - u}\,du \geq \frac{m_f}{x}, \qquad x \in I_j.
\end{equation}
Integrating over the disjoint intervals $I_j$ for $1 \leq j \leq N$ gives
$$
\Vert BH_a^1(f, g_N) \Vert_{L^1(\R)} \geq m_f N \log\left(\frac{1+a_+}{1+a_-}\right).
$$
On the other hand, by the disjointness of the supports of $\psi(\cdot - R_j)$,
$$
\Vert f \Vert_{L^p(\R)} \Vert g_N \Vert_{L^q(\R)} = \Vert f \Vert_{L^p(\R)} \Vert \psi \Vert_{L^q(\R)} N^{1/q}.
$$
Thus we have that
$$
\frac{\Vert BH_a^1(f, g_N) \Vert_{L^1(\R)}}{\Vert f \Vert_{L^p(\R)} \Vert g_N \Vert_{L^q(\R)}} \gtrsim N^{1 - 1/q} = N^{1/p} \to \infty \qquad \text{as } N \to \infty.
$$

In the linear case $\alpha = 1$, a single pair of smooth inputs with nonintegrable output is obtained immediately by setting
\begin{equation}\label{eq:g-inf-linear}
g(y) = \sum_{j \geq 1} \frac{1}{j}\,\psi(y - R_j).
\end{equation}
Since the supports of the bumps are pairwise disjoint and $\sum_{j \geq 1} j^{-q} < \infty$ for every $q > 1$, the function $g$ is smooth and belongs to $\bigcap_{q > 1} L^q(\R)$. However, on each interval $I_j$ we have
$$
BH_a^1(f,g)(x) \geq \frac{m_f}{j x},
$$
so integrating over $\bigcup_{j \geq 1} I_j$ bounds $\Vert BH_a^1(f,g) \Vert_{L^1(\R)}$ below by a positive constant times $\sum_{j \geq 1} j^{-1} = \infty$.

\subsection{The sublinear case: \texorpdfstring{$0 < \alpha < 1$}{0 < alpha < 1}}\label{subsec:alpha-lt-1}
As before, we put $a_* = \frac{1}{2}(A+B)$. Choose a small fixed $\eta > 0$, and then a sufficiently large fixed $L > 0$, depending only on $A$, $B$, and  $\alpha$. Set
$$
R_j = (Lj)^{1/(1-\alpha)}, \qquad w_j = R_j^{\alpha - 1} = \frac{1}{Lj}, \qquad Y_j = R_j + a_* R_j^\alpha.
$$
Define the intervals
$$
I_j = [R_j, R_j + \eta R_j^\alpha], \qquad J_j = (R_j - \eta R_j^\alpha, R_j + 2\eta R_j^\alpha).
$$
On $J_j$, we set
$$
a_j(x) = \frac{Y_j - x}{x^\alpha}.
$$
Writing $x = R_j + \delta R_j^\alpha$ for $\delta \in (-\eta, 2\eta)$ and using $R_j^{\alpha - 1} = (Lj)^{-1}$ gives the relation
$$
a_j(x) = \frac{a_* - \delta}{\left(1 + \delta/(Lj)\right)^\alpha}, \qquad -\eta < \delta < 2\eta.
$$
Choosing $\eta > 0$ sufficiently small and then $L$ sufficiently large ensures that $a_j(J_j)$ lies in a fixed compact subinterval of $(A, B)$ for all $j \geq 1$.

The intervals $J_j$ are pairwise disjoint for $L$ large. Indeed, by convexity of $s \mapsto s^{1/(1-\alpha)}$, for $j \geq 2$ we have
$$
R_j - R_{j-1} = R_j \left[1 - \left(1 - \frac{1}{j}\right)^{1/(1-\alpha)}\right] \geq \frac{R_j}{j} = L R_j^\alpha.
$$
Thus $L > 3\eta$ suffices to separate the intervals $J_j$. Choose a smooth reference cutoff $0 \leq \chi \leq 1$ with $\chi = 1$ on $[0, \eta]$ and $\supp \chi \subset (-\eta, 2\eta)$, and set
$$
\chi_j(x) = \chi\left(\frac{x - R_j}{R_j^\alpha}\right).
$$
As before, we define
$$
a(x) = a_* + \sum_{j \geq 1} \chi_j(x)\left(a_j(x) - a_*\right).
$$
Then $a$ is smooth and takes values in $(A, B)$. Moreover, on $J_j$,
$$
a_j'(x) = -x^{-\alpha} - \frac{\alpha a_j(x)}{x} = O(R_j^{-\alpha}),
$$
and the cutoff derivative term $\chi_j'(x)(a_j(x) - a_*)$ satisfies the same $O(R_j^{-\alpha})$ bound. Since $R_j \to \infty$, $a$ is globally Lipschitz on $\R$.

Fix a nonzero nonnegative function $f \in C_c^\infty((0,1))$ with $m_f = \int_0^1 f(u)\,du > 0$. Choose a nonnegative function $\psi \in C_c^\infty(\R)$ equal to $1$ on $[-B\alpha 2^{1-\alpha}, 0]$.
Define
\begin{equation}\label{eq:gN-sublinear}
g_N(y) = \sum_{j=1}^N \psi\left(\frac{y - Y_j}{w_j}\right).
\end{equation}
The supports of these bumps are pairwise disjoint after increasing $L$ if necessary: the centers $Y_j$ are separated by at least $R_j - R_{j-1} \geq L R_j^\alpha$, whereas the widths are $O(w_j) = O(R_j^{\alpha-1})$. Consequently,
$$
\Vert g_N \Vert_{L^q(\R)} = \Vert \psi \Vert_{L^q(\R)} \left(\sum_{j=1}^N w_j\right)^{1/q}.
$$

For $x \in I_j$ and $0 \leq u \leq 1$, we have $x + a(x)x^\alpha = Y_j$, and the mean value theorem gives
$$
x + a(x)(x - u)^\alpha - Y_j = a(x)\left((x - u)^\alpha - x^\alpha\right) \in [-B\alpha 2^{1-\alpha} w_j, 0],
$$
using $x - u \geq R_j/2$ (which holds for $L$ large) and $\alpha (R_j/2)^{\alpha-1} = \alpha 2^{1-\alpha} w_j$. Thus the $j$th bump of $g_N$ equals $1$ on the relevant arguments. As in the linear case, $f(x-t) \neq 0$ forces $t = x - u > 0$, so every other bump makes a nonnegative contribution and there is no cancellation. Hence, for $L$ large enough that $x \leq 2R_j$ on $I_j$,
\begin{equation}\label{eq:BH-lower-sublinear}
BH_a^\alpha(f, g_N)(x) \geq \int_0^1 \frac{f(u)}{x - u}\,du \geq \frac{m_f}{x} \geq \frac{m_f}{2R_j}, \qquad x \in I_j.
\end{equation}
Integrating over the disjoint intervals $I_j$, which have length $\eta R_j^\alpha$, yields
$$
\Vert BH_a^\alpha(f, g_N) \Vert_{L^1(\R)} \geq \frac{\eta m_f}{2} \sum_{j=1}^N R_j^{\alpha - 1} = \frac{\eta m_f}{2} \sum_{j=1}^N w_j.
$$
Since
$$
\sum_{j=1}^N w_j = L^{-1} \sum_{j=1}^N j^{-1} \simeq_L \log(N),
$$
we have that
$$
\frac{\Vert BH_a^\alpha(f, g_N) \Vert_{L^1(\R)}}{\Vert f \Vert_{L^p(\R)} \Vert g_N \Vert_{L^q(\R)}} \gtrsim \left(\sum_{j=1}^N w_j\right)^{1 - 1/q} \simeq_L \left(\log(N)\right)^{1/p} \to \infty \qquad \text{as } N \to \infty.
$$

\subsection{The superlinear case: \texorpdfstring{$\alpha > 1$}{alpha > 1}}\label{subsec:alpha-gt-1}
When $\alpha > 1$, the shrinking bumps are placed in the first input, and the output intervals lie on the negative half-line. Again we put $a_* = \frac{1}{2}(A+B)$. For a small $\eta \in (0, 1/4)$ and a sufficiently large $L > 0$, set
$$
T_j = (Lj)^{1/(\alpha - 1)}, \qquad w_j = T_j^{1-\alpha} = \frac{1}{Lj}, \qquad X_j = -a_* T_j^\alpha, \qquad U_j = X_j - T_j.
$$
Define
$$
I_j = [X_j, X_j + \eta T_j], \qquad J_j = (X_j - \eta T_j, X_j + 2\eta T_j).
$$
On $J_j$, define
$$
a_j(x) = \frac{-x}{(x - U_j)^\alpha}.
$$
Writing $x = X_j + \delta T_j$ for $\delta \in (-\eta, 2\eta)$ and noting $x - U_j = (1+\delta)T_j$, this becomes
$$
a_j(x) = \frac{a_* - \delta/(Lj)}{(1+\delta)^\alpha}, \qquad -\eta < \delta < 2\eta.
$$
Choose $\eta < 1/4$ small enough and then $L$ large enough that these values lie in a fixed compact subinterval of $(A, B)$ for all $j \geq 1$.

The intervals $J_j$ are pairwise disjoint for $L$ large. In fact, for $j \geq 2$,
$$
X_{j-1} - X_j = a_*(T_j^\alpha - T_{j-1}^\alpha) \geq a_* \frac{T_j^\alpha}{j} = a_* L T_j.
$$
Choose smooth cutoffs
$$
\chi_j(x) = \chi\left(\frac{x - X_j}{T_j}\right)
$$
which are supported in $J_j$ and equal to $1$ on $I_j$. We then define $a$ using the same expression as \eqref{eq:a-def-linear}. On $J_j$,
$$
a_j'(x) = -(x - U_j)^{-\alpha} - \frac{\alpha a_j(x)}{x - U_j} = O(T_j^{-1}),
$$
and the cutoff derivative has the same $O(T_j^{-1})$ bound. This again gives that $a:\R \to (A,B)$ is smooth and globally Lipschitz.

Choose a nonzero nonnegative function $\psi \in C_c^\infty((-1,1))$ and write $m_\psi = \int_{-1}^1 \psi(v)\,dv > 0$. Let
$$
G = 1 + B\alpha 2^{\alpha - 1},
$$
and choose $g \in C_c^\infty((-2G, 2G))$ with $0 \leq g \leq 1$ and $g = 1$ on $[-G, G]$. Define
\begin{equation}\label{eq:fN-superlinear}
f_N(u) = \sum_{j=1}^N \psi\left(\frac{u - U_j}{w_j}\right).
\end{equation}
The bump supports are pairwise disjoint for $L$ large, and
$$
\Vert f_N \Vert_{L^p(\R)} = \Vert \psi \Vert_{L^p(\R)} \left(\sum_{j=1}^N w_j\right)^{1/p}.
$$

Fix $x \in I_j$ and set $t_0 = x - U_j \in [T_j, (1+\eta)T_j]$. By definition of $a_j(x)$, we have $x + a(x)t_0^\alpha = 0$. For $|v| < 1$, put $u = U_j + w_j v$ and
$$
t = x - u = t_0 - w_j v.
$$
For $L$ sufficiently large, $T_j/2 < t < 2T_j$, and by the mean value theorem,
$$
\left|x + a(x)t^\alpha\right| = a(x)\left|(t_0 - w_j v)^\alpha - t_0^\alpha\right| \leq B\alpha(2T_j)^{\alpha - 1} w_j = B\alpha 2^{\alpha - 1} < G.
$$
Thus $g(x + a(x)t^\alpha) = 1$ throughout the contribution of the $j$th bump.

We now verify that all other bumps $k \neq j$ and all $t < 0$ contributions vanish for $x \in I_j$. For the odd power convention $t^\alpha = \sgn(t)|t|^\alpha$, any $t < 0$ gives
$$
x + a(x)t^\alpha < x \leq X_1 + \eta T_1 < -2G
$$
for $L$ large enough, so $g(x + a(x)t^\alpha) = 0$. For the even power convention $t^\alpha = |t|^\alpha$, any nonzero contribution of $g(x + a(x)t^\alpha)$ requires $|x + a(x)|t|^\alpha| < 2G$. Since $x = -a(x)t_0^\alpha$, this implies
$$
\left||t|^\alpha - t_0^\alpha\right| \leq \frac{2G}{A}.
$$
For $L$ sufficiently large, this forces $|t| \leq 2t_0 < 3T_j$, and the same upper bound applies to $t > 0$ under either convention. Now compare the support of the $k$th bump with $|x - u| = |t| \leq 3T_j$: if $k < j$, then
$$
U_k - X_j = a_*(T_j^\alpha - T_k^\alpha) - T_k \geq (a_* L - 1)T_j,
$$
whereas if $k > j$, then
$$
X_j - U_k = a_*(T_k^\alpha - T_j^\alpha) + T_k \geq a_* L T_k \geq a_* L T_j.
$$
Taking $a_* L$ sufficiently large and noting $w_k \leq 1/L$ excludes every bump $k \neq j$ from the region $|x - u| \leq 3T_j$. As for the $j$th bump itself, every $u \in (U_j - w_j, U_j + w_j)$ satisfies $u < x$, so it contributes only for $t = x - u > 0$. Therefore, on $I_j$, the full integral defining $BH_a^\alpha(f_N, g)(x)$ consists solely of the positive contribution of the $j$th bump:
\begin{equation}\label{eq:BH-lower-superlinear}
BH_a^\alpha(f_N, g)(x) = w_j \int_{-1}^1 \frac{\psi(v)}{t_0 - w_j v}\,dv \geq \frac{m_\psi w_j}{2T_j}, \qquad x \in I_j.
\end{equation}
Integrating over the disjoint intervals $I_j$ of length $\eta T_j$ gives
$$
\Vert BH_a^\alpha(f_N, g) \Vert_{L^1(\R)} \geq \frac{\eta m_\psi}{2} \sum_{j=1}^N w_j.
$$
So
$$
\frac{\Vert BH_a^\alpha(f_N, g) \Vert_{L^1(\R)}}{\Vert f_N \Vert_{L^p(\R)} \Vert g \Vert_{L^q(\R)}} \gtrsim \left(\sum_{j=1}^N w_j\right)^{1 - 1/p} \simeq_L \left(\log(N)\right)^{1/q} \to \infty \qquad \text{as } N \to \infty.
$$
This establishes \eqref{eq:BH-norm-failure} of Theorem \ref{thm:main1} in every regime.

\subsection{One pair of smooth \texorpdfstring{$L^s$}{Ls} inputs with nonintegrable output}\label{subsec:infinite-train}
For either nonlinear regime ($0 < \alpha < 1$ or $\alpha > 1$), replace the finite bump train by the infinite bump train weighted by the heights $\left(\log(j+1)\right)^{-1}$, $j \geq 1$. Specifically, for $0 < \alpha < 1$, keep $f \in C_c^\infty((0,1))$ fixed and set
\begin{equation}\label{eq:g-inf-sublinear}
g(y) = \sum_{j \geq 1} \frac{1}{\log(j+1)} \,\psi\left(\frac{y - Y_j}{w_j}\right),
\end{equation}
whereas for $\alpha > 1$, keep $g \in C_c^\infty((-2G, 2G))$ fixed and set
\begin{equation}\label{eq:f-inf-superlinear}
f(u) = \sum_{j \geq 1} \frac{1}{\log(j+1)} \,\psi\left(\frac{u - U_j}{w_j}\right).
\end{equation}
In both cases, the supports of the individual bumps are pairwise disjoint, so the resulting infinite--train function is $C^\infty(\R)$. Since $w_j = 1/(Lj)$, for every $s > 1$ we have
$$
\sum_{j \geq 1} \frac{w_j}{\left(\log(j+1)\right)^s} = \frac{1}{L} \sum_{j \geq 1} \frac{1}{j \left(\log(j+1)\right)^s} < \infty,
$$
whereas
$$
\sum_{j \geq 1} \frac{w_j}{\log(j+1)} = \frac{1}{L} \sum_{j \geq 1} \frac{1}{j \log(j+1)} = \infty.
$$
Thus the infinite bump train belongs to $\bigcap_{s > 1} L^s(\R)$, while the same interval lower bounds \eqref{eq:BH-lower-sublinear} and \eqref{eq:BH-lower-superlinear} on $I_j$, now multiplied by $(\log(j+1))^{-1}$, show that
$$
\Vert BH_a^\alpha(f, g) \Vert_{L^1(\R)} \geq \sum_{j \geq 1} \int_{I_j} \left|BH_a^\alpha(f,g)(x)\right|\,dx \gtrsim \sum_{j \geq 1} w_j \frac{1}{\log(j+1)} = \infty.
$$
The linear case $\alpha = 1$ was already handled in \eqref{eq:g-inf-linear} with the weight $j^{-1}$ and constant bump widths.

Note that although these infinite-train inputs are smooth and belong to $L^s(\R)$ for all $s > 1$, they are neither compactly supported nor lying in the Schwartz class. Nevertheless, the principal value integral defining $BH_a^\alpha(f,g)(x)$ is well-defined at every $x \in \R$: in the $\alpha \leq 1$ constructions, the compact support of $f$ restricts $t$ to a bounded interval for each $x$, while in the $\alpha > 1$ construction, the compact support of $g$ together with $a(x) \geq A > 0$ restricts $t$ to a bounded interval for each $x$. Near $t = 0$, the smoothness of $f$ and $g$ and the local $\min(1, \alpha)$-H\"{o}lder continuity of $t \mapsto t^\alpha$ ensure absolute convergence of the symmetric principal value. This completes the proof of Theorem \ref{thm:main1}. \qed

\section{Proof of Theorem \ref{thm:main2}}\label{sec:proof-thm2}
We break the proof up into three steps: in Step 1 we establish the negative results at the $r=1$ endpoint by adapting the positive bounded coefficient construction from Section \ref{sec:proof-thm1} to $BM_a^\alpha$ (including the endpoint exponent pairs where $p = \infty$ or $q = \infty$); in Step 2 we prove the strong--type $L^p \times L^q \to L^r$ bounds for $r > 1$; and in Step 3 we establish the weak--type $L^p \times L^q \to L^{1,\infty}$ bounds for $r = 1$.

\medskip
\noindent\texttt{Step 1: Endpoint failure of $BM_a^\alpha : L^p \times L^q \to L^1$ for $a : \R \to (A,B)$: }\\
Fix $\alpha > 0$, $0 < A < B < \infty$, and H\"{o}lder conjugate exponents $1 \leq p, q \leq \infty$ with $1/p + 1/q = 1$. We first treat $1 < p, q < \infty$ (together with the endpoint cases directly covered by the disjoint bump trains) using the exact same smooth, globally Lipschitz coefficient $a : \R \to (A, B)$ and test functions constructed in Section \ref{sec:proof-thm1}, and we then address the remaining infinite-exponent cases $p = \infty$ or $q = \infty$.

\begin{itemize}
\item \textbf{Case 1: $\alpha = 1$.} Let $a : \R \to (A,B)$, $f \in C_c^\infty((0,1))$, and $g_N$ be as in Subsection \ref{subsec:alpha-eq-1}. For each $x \in I_j = \left[\frac{R_j}{1+a_+}, \frac{R_j}{1+a_-}\right]$ ($1 \leq j \leq N$), choose the averaging scale $r = x > 0$ in the supremum defining $BM_a^1(f, g_N)(x)$. Since $x > 2$ and $u = x - t \in (0,1)$ implies that
$$
t = x - u \in (x-1, x) \subset [-x, x],
$$
changing variables to $u = x - t$ gives for $x \in I_j$ that
$$
BM_a^1(f, g_N)(x) \geq \frac{1}{2x} \int_{-x}^x f(x-t) g_N\left(x + a(x)t\right)\,dt \geq \frac{1}{2x} \int_0^1 f(u) g_N\left(x + a(x)(x-u)\right)\,du = \frac{m_f}{2x},
$$
because
$$
x + a(x)(x-u) = R_j - a(x)u \in [R_j - B, R_j]
$$
where the $j$th bump of $g_N$ equals $1$. Integrating over the disjoint intervals $I_j$ for $1 \leq j \leq N$ yields
$$
\Vert BM_a^1(f, g_N) \Vert_{L^1(\R)} \geq \frac{m_f}{2} N \log\left(\frac{1+a_+}{1+a_-}\right),
$$
while
$$
\Vert f \Vert_{L^p(\R)} \Vert g_N \Vert_{L^q(\R)} = \Vert f \Vert_{L^p(\R)} \Vert \psi \Vert_{L^q(\R)} N^{1/q}
$$
for all $1 \leq q \leq \infty$. Thus
$$
\frac{\Vert BM_a^1(f,g_n)\Vert_{L^1(\R)}}{\Vert f \Vert_{L^p(\R)} \Vert g_N \Vert_{L^q(\R)}} \gtrsim N^{1 - 1/q} = N^{1/p} \to \infty
$$
for all $1 \leq p < \infty$, including $(p,q) = (1,\infty)$. Similarly, for the single pair $(f, g)$ with $g$ given by the infinite bump train \eqref{eq:g-inf-linear}, we have
$$
BM_a^1(f,g)(x) \geq \frac{m_f}{2jx}
$$
on $I_j$, so $\Vert BM_a^1(f,g) \Vert_{L^1(\R)} = \infty$.

\item \textbf{Case 2: $0 < \alpha < 1$.} Let $a : \R \to (A,B)$, $f \in C_c^\infty((0,1))$, and $g_N$ be as in Subsection \ref{subsec:alpha-lt-1}. For $x \in I_j = [R_j, R_j + \eta R_j^\alpha]$ ($1 \leq j \leq N$), we again choose $r = x > 0$ in the supremum defining $BM_a^\alpha(f, g_N)(x)$. For $u = x - t \in (0,1)$, we saw in Subsection \ref{subsec:alpha-lt-1} that
$$
x + a(x)(x-u)^\alpha - Y_j \in [-B\alpha 2^{1-\alpha} w_j, 0],
$$
so the $j$th bump of $g_N$ equals $1$. Since $x \leq 2R_j$ on $I_j$, we obtain
$$
BM_a^\alpha(f, g_N)(x) \geq \frac{1}{2x} \int_0^1 f(u) g_N\left(x + a(x)(x-u)^\alpha\right)\,du = \frac{m_f}{2x} \geq \frac{m_f}{4R_j}, \qquad x \in I_j.
$$
Integrating over the disjoint intervals $I_j$ of length $\eta R_j^\alpha$ gives
$$
\Vert BM_a^\alpha(f, g_N) \Vert_{L^1(\R)} \geq \frac{\eta m_f}{4} \sum_{j=1}^N R_j^{\alpha - 1} = \frac{\eta m_f}{4} \sum_{j=1}^N w_j \simeq_L \log(N).
$$
Since the bumps of $g_N$ have disjoint supports,
$$
\Vert g_N \Vert_{L^q(\R)} = \Vert \psi \Vert_{L^q(\R)} \left(\sum_{j=1}^N w_j\right)^{1/q}
$$
for all $1 \leq q \leq \infty$, so 
$$
\frac{\Vert BM_a^\alpha(f,g_n)\Vert_{L^1(\R)}}{\Vert f \Vert_{L^p(\R)} \Vert g_N \Vert_{L^q(\R)}} \gtrsim \left(\log(N)\right)^{1 - 1/q} = \left(\log(N)\right)^{1/p} \to \infty
$$
for all $1 \leq p < \infty$, again including $(p,q) = (1,\infty)$. Likewise, replacing $g_N$ by the infinite bump train $g \in \bigcap_{s>1}L^s(\R)$ from \eqref{eq:g-inf-sublinear} gives
$$
BM_a^\alpha(f,g)(x) \geq \frac{m_f}{4R_j\log(j+1)}
$$
on $I_j$, and hence
$$
\Vert BM_a^\alpha(f,g) \Vert_{L^1(\R)} \gtrsim \sum_{j \geq 1} \frac{w_j}{\log(j+1)} = \infty.
$$

\item \textbf{Case 3: $\alpha > 1$.} Let $a : \R \to (A,B)$, $f_N$, and $g \in C_c^\infty((-2G, 2G))$ be as in Subsection \ref{subsec:alpha-gt-1}. For $x \in I_j = [X_j, X_j + \eta T_j]$ ($1 \leq j \leq N$), recall that
$$
t_0 = x - U_j \in [T_j, (1+\eta)T_j],
$$
and for $u = U_j + w_j v$ with $|v| < 1$, we have
$$
t = x - u = t_0 - w_j v \in \left(\frac{T_j}{2}, 2T_j\right)
$$
and $g(x + a(x)t^\alpha) = 1$. Choosing the averaging scale $r = 2T_j$ in the supremum defining $BM_a^\alpha(f_N, g)(x)$ and using the nonnegativity of the integrand to restrict to the $j$th bump of $f_N$, we obtain
$$
BM_a^\alpha(f_N, g)(x) \geq \frac{1}{4T_j} \int_{-2T_j}^{2T_j} f_N(x-t) g\left(x + a(x)t^\alpha\right)\,dt \geq \frac{w_j}{4T_j} \int_{-1}^1 \psi(v)\,dv = \frac{m_\psi w_j}{4T_j}, \qquad x \in I_j.
$$
Integrating over the disjoint intervals $I_j$ of length $\eta T_j$ gives
$$
\Vert BM_a^\alpha(f_N, g) \Vert_{L^1(\R)} \geq \frac{\eta m_\psi}{4} \sum_{j=1}^N w_j \simeq_L \log(N+1).
$$
Since
$$
\Vert f_N \Vert_{L^p(\R)} = \Vert \psi \Vert_{L^p(\R)} \left(\sum_{j=1}^N w_j\right)^{1/p}
$$
for all $1 \leq p \leq \infty$, we once again have
$$
\frac{\Vert BM_a^\alpha(f,g_n)\Vert_{L^1(\R)}}{\Vert f \Vert_{L^p(\R)} \Vert g_N \Vert_{L^q(\R)}} \gtrsim \left(\log(N)\right)^{1 - 1/p} = \left(\log(N)\right)^{1/q} \to \infty
$$
for all $1 \leq q < \infty$, including $(p,q) = (\infty, 1)$. Finally, replacing $f_N$ by the infinite bump train $f \in \bigcap_{s>1}L^s(\R)$ from \eqref{eq:f-inf-superlinear} gives
$$
BM_a^\alpha(f,g)(x) \geq \frac{m_\psi w_j}{4T_j\log(j+1)}
$$ 
on $I_j$, and therefore
$$
\Vert BM_a^\alpha(f,g) \Vert_{L^1(\R)} \gtrsim \sum_{j \geq 1} \frac{w_j}{\log(j+1)} = \infty.
$$

\item \textbf{The remaining endpoint cases $(p,q) = (1,\infty)$ and $(p,q) = (\infty,1)$.} As noted above, the disjointness of the bump supports in $g_N$ (for $\alpha \leq 1$) and $f_N$ (for $\alpha > 1$) already implies $\Vert g_N \Vert_{L^\infty(\R)} = \Vert \psi \Vert_{L^\infty(\R)}$ and $\Vert f_N \Vert_{L^\infty(\R)} = \Vert \psi \Vert_{L^\infty(\R)}$, which immediately handles $(p,q) = (1,\infty)$ when $\alpha \leq 1$ and $(p,q) = (\infty,1)$ when $\alpha > 1$. In fact, for any $\alpha > 0$ and any measurable $a : \R \to (A,B)$, both endpoint pairs $(p,q) = (1,\infty)$ and $(p,q) = (\infty,1)$ also fail by a direct flat--top cutoff argument:
\begin{enumerate}[(i)]
\item For $(p,q) = (1,\infty)$ and any $\alpha > 0$, fix a nonzero nonnegative $f \in C_c^\infty((0,1))$ with $m_f = \Vert f \Vert_{L^1(\R)} > 0$, and for $M > 2$ choose $g_M \in C_c^\infty(\R)$ satisfying $0 \leq g_M \leq 1$ and $g_M = 1$ on $[0, M + B M^\alpha]$. For every $x \in (2, M)$ and $u = x - t \in (0,1)$, we have
$$
t = x - u \in (0, M)
$$
and
$$
x + a(x)t^\alpha \in [0, M + B M^\alpha],
$$
so $g_M(x + a(x)t^\alpha) = 1$. Choosing $r = x$ gives 
$$
BM_a^\alpha(f, g_M)(x) \geq \frac{m_f}{2x}
$$ 
on $(2, M)$, and hence
$$
\frac{\Vert BM_a^\alpha(f, g_M) \Vert_{L^1(\R)}}{\Vert f \Vert_{L^1(\R)} \Vert g_M \Vert_{L^\infty(\R)}} \geq \frac{1}{2}\log\left(\frac{M}{2}\right) \to \infty \qquad \text{as } M \to \infty.
$$
\item For $(p,q) = (\infty,1)$ and any $\alpha > 0$, fix a nonzero nonnegative $g \in C_c^\infty((0,1))$ and, for $M > 2$, choose $f_M \in C_c^\infty(\R)$ satisfying $0 \leq f_M \leq 1$ and $f_M = 1$ on $\left[-(2M/A)^{1/\alpha} - M, 0\right]$. For each $x \in (-M, -2)$, set 
$$
r_x = \left(\frac{1-x}{A}\right)^{1/\alpha} \simeq_{A,\alpha} |x|^{1/\alpha}.
$$
When $t > 0$ is such that $y := x + a(x)t^\alpha \in (0,1)$, we have
$$
t = \left(\frac{y-x}{a(x)}\right)^{1/\alpha} \in (0, r_x),
$$
$$
x - t \in \left[-(2M/A)^{1/\alpha} - M, 0\right] \implies f_M(x-t) = 1,
$$
and
$$
\frac{dt}{dy} = \frac{(y-x)^{\frac{1}{\alpha}-1}}{\alpha a(x)^{1/\alpha}} \simeq_{A,B,\alpha} |x|^{\frac{1}{\alpha}-1}.
$$ 
Consequently,
$$
BM_a^\alpha(f_M, g)(x) \geq \frac{1}{2r_x} \int_0^{r_x} f_M(x-t) g\left(x + a(x)t^\alpha\right)\,dt \gtrsim_{A,B,\alpha} \frac{\Vert g \Vert_{L^1(\R)}}{|x|}, \qquad -M < x < -2.
$$
Integrating over $x \in (-M, -2)$ while noting that $\Vert f_M \Vert_{L^\infty(\R)} = 1$ yields
$$
\frac{\Vert BM_a^\alpha(f_M, g) \Vert_{L^1(\R)}}{\Vert f_M \Vert_{L^\infty(\R)} \Vert g \Vert_{L^1(\R)}} \gtrsim_{A,B,\alpha} \log\left(\frac{M}{2}\right) \to \infty \qquad \text{as } M \to \infty.
$$
\end{enumerate}
\end{itemize}
This establishes all negative assertions of Theorem \ref{thm:main2}.

\medskip
\noindent\texttt{Step 2: $BM_a^\alpha$ maps $L^p \times L^q \to L^r$ boundedly for $r > 1$: }

We require the following lemma.

\begin{lemma}\label{lem1}
    Fix $\alpha >0$, $a \in \R$, and $R > 0$. For all $h \in L^1_{\mathrm{loc}}(\R)$ we have
    $$
    \frac{1}{2R} \int_{-R}^R \left|h(x + at^\alpha)\right|\, dt \leq C_\alpha Mh(x),
    $$
    where $M$ is the standard centered Hardy--Littlewood maximal function. The constant $C_\alpha$ depends only on $\alpha$ and is independent of $a$ and $R$.
\end{lemma}
\begin{proof}
    If $a = 0$ the result is immediate. The proof for $a < 0$ is analogous to the case $a > 0$, so we focus on the case $a > 0$. By symmetry, it suffices to handle the integral over $t \in [0, R]$ and prove that
    $$
    \frac{1}{R} \int_{0}^R \left|h(x + at^\alpha)\right|\, dt \leq C_\alpha Mh(x).
    $$
    Making the change of variables $u = at^\alpha$ (which we relabel as $t$) gives
    $$
    \frac{1}{R} \int_{0}^R \left|h(x + at^\alpha)\right|\, dt = \int_0^{aR^\alpha} |h(x+t)| \left(\frac{t^{\frac{1}{\alpha}-1}}{\alpha a^{\frac{1}{\alpha}} R} \right)\, dt.
    $$
    Denote
    $$
    w(t) = \frac{t^{\frac{1}{\alpha}-1}}{\alpha a^{\frac{1}{\alpha}} R}.
    $$
    Then $w$ is a probability density on the interval $[0,aR^\alpha]$. We consider two cases according to the size of $\alpha$:
    \begin{itemize}
        \item \textbf{Case 1:} $0 < \alpha \leq 1$.\\
        In this regime $w$ is nondecreasing on $(0, aR^\alpha]$, and hence
        $$
        \int_0^{aR^\alpha} |h(x+t)| w(t)\, dt \leq w(aR^\alpha) \int_0^{aR^\alpha} |h(x+t)|\, dt = \frac{1}{\alpha aR^\alpha} \int_0^{aR^\alpha} |h(x+t)|\, dt \leq \frac{2}{\alpha} Mh(x).
        $$
        \item \textbf{Case 2:} $\alpha >1$.\\
        Here $w$ is strictly decreasing on $(0, aR^\alpha]$. Define a positive Borel measure $\mu$ on $(0,aR^\alpha]$ by
        $$
        d\mu(s) = -sw'(s)\, ds + \frac{1}{\alpha} \delta_{aR^\alpha}(ds) = -s\, dw(s) + \frac{1}{\alpha} \delta_{aR^\alpha}(ds).
        $$
        By the fundamental theorem of calculus, for every $t \in (0, aR^\alpha]$,
        $$
        w(t) = \int_t^{aR^\alpha} \frac{1}{s}\, d\mu(s),
        $$
        and integrating $t$ over $(0, aR^\alpha]$ shows that $\mu((0, aR^\alpha]) = \int_0^{aR^\alpha} w(t)\,dt = 1$, so $\mu$ is a probability measure. Applying Tonelli's theorem, we obtain
        \begin{equation*}
            \begin{split}
                \int_0^{aR^\alpha} |h(x+t)|w(t)\, dt &= \int_0^{aR^\alpha} \int_t^{aR^\alpha} \frac{1}{s} |h(x+t)|\, d\mu(s)\, dt\\
                &= \int_0^{aR^\alpha} \frac{1}{s}\int_0^s |h(x+t)|\, dt\, d\mu(s)\\
                &\leq 2Mh(x) \int_0^{aR^\alpha} d\mu(s)\\
                &= 2Mh(x).
            \end{split}
        \end{equation*}
    \end{itemize}
    Thus, taking $C_\alpha = 2 \max(1,\alpha^{-1})$ yields the desired inequality.
\end{proof}

Now fix $r >0$. For any H\"{o}lder conjugate exponents $s, s' > 1$, H\"{o}lder's inequality gives
$$
\frac{1}{2r} \int_{-r}^r \left|f(x-t) g\left(x+a(x)t^\alpha\right)\right|\, dt \leq \left(\frac{1}{2r} \int_{-r}^r |f(x-t)|^s\, dt\right)^{1/s} \left(\frac{1}{2r} \int_{-r}^r \left|g\left(x+a(x)t^\alpha\right)\right|^{s'}\, dt\right)^{1/s'}.
$$
Applying Lemma \ref{lem1} at each point $x$ with $a = a(x)$ and taking the supremum over $r > 0$ yields the pointwise estimate
\begin{equation}\label{eq:BM_pointwsie_est}
    BM_a^\alpha(f,g)(x) \lesssim_{\alpha,s} \left(M(|f|^s)(x)\right)^{1/s} \left(M(|g|^{s'})(x)\right)^{1/s'}.
\end{equation}
Now fix $1 < p, q \leq \infty$ and $r > 1$ satisfying
$$
\frac{1}{p} + \frac{1}{q} = \frac{1}{r}.
$$
When $p, q < \infty$, let $p'$ and $q'$ denote the H\"{o}lder conjugates of $p$ and $q$. Since $r > 1$, we have $q' < p$ and $p' < q$. Thus we may choose H\"{o}lder conjugate exponents $s, s' > 1$ such that
$$
q' < s < p \qquad \text{and} \qquad p' < s' < q.
$$
Applying \eqref{eq:BM_pointwsie_est}, H\"{o}lder's inequality with exponents $p/r$ and $q/r$, and the strong--type $L^{p/s}$ and $L^{q/s'}$ boundedness of the Hardy--Littlewood maximal function (since $p/s > 1$ and $q/s' > 1$), we obtain
\begin{equation*}
    \begin{split}
        \Vert BM_a^\alpha(f,g) \Vert_{L^r(\R)} &\lesssim_{\alpha,s} \left( \int_{\R} \left|\left(M(|f|^s)(x)\right)^{1/s} \left(M(|g|^{s'})(x)\right)^{1/s'}\right|^r\, dx\right)^{1/r}\\
        &\leq \left( \int_{\R} \left|M(|f|^s)(x)\right|^{p/s}\, dx\right)^{1/p} \left( \int_{\R} \left|M(|g|^{s'})(x)\right|^{q/s'}\, dx\right)^{1/q}\\
        &\lesssim_{\alpha,s,p,q} \Vert f \Vert_{L^p(\R)} \Vert g \Vert_{L^q(\R)}.
    \end{split}
\end{equation*}
When $p = \infty$ (so $q = r > 1$) or $q = \infty$ (so $p = r > 1$), one simply bounds the $L^\infty$ factor out of the average and applies Lemma \ref{lem1} together with the $L^r$ boundedness of $M$ directly.

\medskip
\noindent\texttt{Step 3: $BM_a^\alpha$ maps $L^p \times L^q \to L^{1,\infty}$ boundedly for $1/p + 1/q = 1$: }

First, if $p = 1$ and $q = \infty$, then
$$
BM_a^\alpha(f,g)(x) \leq \Vert g \Vert_{L^\infty(\R)} \sup_{r > 0} \frac{1}{2r} \int_{-r}^r |f(x-t)|\, dt = Mf(x) \Vert g \Vert_{L^\infty(\R)}. 
$$
Applying the weak--type $(1,1)$ estimate for the Hardy--Littlewood maximal function gives
$$
\Vert BM_a^\alpha(f,g) \Vert_{L^{1,\infty}(\R)} \lesssim \Vert f \Vert_{L^1(\R)} \Vert g \Vert_{L^\infty(\R)}.
$$
Similarly, if $p = \infty$ and $q = 1$, Lemma \ref{lem1} yields
$$
BM_a^\alpha(f,g)(x) \leq \Vert f \Vert_{L^\infty(\R)} \sup_{r > 0} \frac{1}{2r} \int_{-r}^r \left|g\left(x + a(x)t^\alpha\right)\right|\, dt \lesssim_\alpha Mg(x) \Vert f \Vert_{L^\infty(\R)},
$$
and hence
$$
\Vert BM_a^\alpha(f,g) \Vert_{L^{1,\infty}(\R)} \lesssim_\alpha \Vert f \Vert_{L^\infty(\R)} \Vert g \Vert_{L^1(\R)}.
$$
To handle the remaining cases, fix a H\"{o}lder conjugate pair $1 < p, q < \infty$ and assuming without loss of generality that $\Vert f \Vert_{L^p(\R)}, \Vert g \Vert_{L^q(\R)} > 0$ set, for $\lambda > 0$,
$$
E_\lambda = \left\{x \in \R : BM_a^\alpha(f,g)(x) > \lambda\right\}.
$$
Applying \eqref{eq:BM_pointwsie_est} with $s = p$ and $s' = q$ and using the weak--type $(1,1)$ bound for $M$, we have for every $\mu > 0$
\begin{equation}\label{eq:weak_type_est}
        |E_\lambda| \leq \left|\{x : M(|f|^p)(x) > \mu\}\right| + \left|\left\{x : M(|g|^q)(x) > \frac{(\lambda/C_{\alpha,p})^q}{\mu^{q/p}}\right\}\right| \lesssim_{\alpha,p,q} \frac{\Vert f \Vert_{L^p(\R)}^p}{\mu} + \frac{\mu^{q/p} \Vert g \Vert_{L^q(\R)}^q}{\lambda^q}.
\end{equation}
Minimizing the right-hand side of \eqref{eq:weak_type_est} with respect to $\mu > 0$ by choosing
$$
\mu = \left( \frac{p}{q} \lambda^q \frac{\Vert f \Vert_{L^p(\R)}^p}{\Vert g \Vert_{L^q(\R)}^q}\right)^{\frac{p}{p+q}}
$$
and using $1/p + 1/q = 1$ gives
$$
\lambda |E_\lambda| \lesssim_{\alpha,p,q} \Vert f \Vert_{L^p(\R)} \Vert g \Vert_{L^q(\R)}.
$$
Taking the supremum over $\lambda > 0$ completes the proof of Theorem \ref{thm:main2}. \qed

\section{Proof of Proposition \ref{prop:unit-interval-shift}}
Choose a nonnegative function $g\in C_c^\infty(\mathbb R)$ which is identically $1$ on $[0,1+2^\alpha]$. For each integer $N \geq 2$, set
$\varepsilon_N:=e^{-N^2}$,
and choose a nonnegative function $f_N\in C_c^\infty((-1,0))$
satisfying $0 \leq f_N \leq 1$ and
$$
f_N\equiv1
\qquad\text{on }
\left[-\frac12,-\varepsilon_N\right].
$$
Since the functions $f_N$ are uniformly bounded and supported in a
fixed compact interval, we have
\begin{equation}\label{eq:uniform-fN-norm}
\sup_{N\geq2}\|f_N\|_{L^p(\mathbb R)}
\lesssim_p 1.
\end{equation}
Moreover, $\|g\|_{L^q(\mathbb R)} < \infty$ for every $1 < q \leq \infty$. We claim that $BSH_{[0,1]}^\alpha(f_N,g)$ becomes arbitrarily large on an interval of length comparable to $N^{-1}$. To see this, for $0 < x < N^{-1}$ put
$$
c_x = x - e^{-1/x^2}.
$$
Since $0<x<1/2$, we have
$$
0 < e^{-1/x^2} < x,
$$
and hence
$$
0 < c_x < x < 1.
$$
Thus $c_x \in [0,1]$, and by the definition of
$BSH_{[0,1]}^\alpha$,
\begin{equation}\label{eq:choose-cx}
BSH_{[0,1]}^\alpha(f_N,g)(x)
\geq
\left|
 \int_{\mathbb R}
f_N(x-t)g(x+t^\alpha)\,
\frac{dt}{t-c_x}
\right|.
\end{equation}
Notice that for $f_N(x-t)$ to be nonzero we must have $t > x > c_x$, so the integrand in \eqref{eq:choose-cx} is nonnegative and supported strictly to the right of the singularity $t = c_x$. We may therefore restrict the integral in \eqref{eq:choose-cx} to
$x + \varepsilon_N \leq t \leq x + \frac12.$
Write
$$
t = x + s,
\qquad
\varepsilon_N \leq s \leq \frac12.
$$
Then
$$
x-t=-s
\in
\left[-\frac12,-\varepsilon_N\right],
$$
so that by construction $f_N(x-t) = 1$. Also,
$$
0<t=x+s<1,
$$
and therefore
$$
0 < x + t^\alpha < 1 + 2^\alpha.
$$
By our choice of $g$, we have $g(x+t^\alpha)=1$. Finally, since
$$
t - c_x = x + s - \left(x-e^{-1/x^2}\right) = s + e^{-1/x^2}
$$
and $0 < x < 1/N$, we have
$$
e^{-1/x^2} \leq e^{-N^2} = \varepsilon_N.
$$
Hence
$$
BSH_{[0,1]}^\alpha(f_N,g)(x) \geq
\int_{\varepsilon_N}^{1/2}
\frac{ds}{s+e^{-1/x^2}} \geq
\int_{\varepsilon_N}^{1/2}
\frac{ds}{s+\varepsilon_N}.
$$
For $s \geq \varepsilon_N$ we trivially have $s + \varepsilon_N \leq 2s$. Consequently,
$$
BSH_{[0,1]}^\alpha(f_N,g)(x)
\geq
\frac12
\int_{\varepsilon_N}^{1/2}\frac{ds}{s}
= \frac12 \log\left(\frac{1}{2\varepsilon_N}\right) = \frac12 \left(N^2 - \log2 \right).
$$
Thus for all sufficiently large $N$ we have
\begin{equation}\label{eq:unit-interval-pointwise-lower}
BSH_{[0,1]}^\alpha(f_N,g)(x)
\geq
\frac{N^2}{4},
\qquad
0<x<\frac1N.
\end{equation}
Set $\lambda_N := \frac{N^2}{8}$. It follows from \eqref{eq:unit-interval-pointwise-lower} that
$$
\left(0,\frac1N\right)
\subseteq
\left\{
x\in\mathbb R:
BSH_{[0,1]}^\alpha(f_N,g)(x)>\lambda_N
\right\},
$$
and therefore
$$
\left|
\left\{
x\in\mathbb R:
BSH_{[0,1]}^\alpha(f_N,g)(x)>\lambda_N
\right\}
\right|
\geq
\frac1N.
$$
By the definition of the weak $L^r$ quasi-norm,
$$
\|BSH_{[0,1]}^\alpha(f_N,g)\|_{L^{r,\infty}(\R)} \geq \lambda_N \left(\frac1N\right)^{1/r} = \frac18 N^{2 - \frac1r}.
$$
Since $r>\frac12$, we have $2-\frac1r>0$, and hence
\begin{equation}\label{eq:weak-norm-diverges}
\|BSH_{[0,1]}^\alpha(f_N,g)\|_{L^{r,\infty}(\R)} \to \infty
\qquad\text{as }N\to\infty.
\end{equation}
On the other hand, by \eqref{eq:uniform-fN-norm},
$$
\sup_{N\geq2}
\|f_N\|_{L^p(\R)} \|g\|_{L^q(\R)} < \infty.
$$
Combining this with \eqref{eq:weak-norm-diverges}, we obtain
$$
\frac{
\|BSH_{[0,1]}^\alpha(f_N,g)\|_{L^{r,\infty}(\R)}
}{
\|f_N\|_{L^p(\R)}\|g\|_{L^q(\R)}
}\to \infty \qquad \text{as } N \to \infty.
$$
Thus no estimate of the form
\eqref{eq:unit-interval-weak-bound} can hold. Finally, since
$$
L^r(\mathbb R)\subseteq L^{r,\infty}(\mathbb R)
\qquad (0<r<\infty),
$$
the corresponding strong--type estimate
$$
\|BSH_{[0,1]}^\alpha(f,g)\|_{L^r(\R)} \lesssim \|f\|_{L^p(\R)}\|g\|_{L^q(\R)}
$$
fails as well. \qed

\section{Proof of Proposition \ref{prop:dyadic}}\label{sec:proof-dyadic}

Fix $\alpha>0$ and let $1\leq p,q\leq\infty$ satisfy
$1/r=1/p+1/q$. We start by defining a rapidly growing sequence of
dyadic scales. Let $K=K(\alpha)\geq4$ be a sufficiently large integer and set
$$
 c_j=2^{Kj},\qquad j\geq1.
$$
We choose $K$ large enough that $2^{K\alpha}\geq16$ and all scale-separation inequalities below
hold for every $j\geq1$. Define the measurable function $c:\R\to\D$ by
\begin{equation}\label{eq:fixed-dyadic-selector}
 c(x)=
 \begin{cases}
 c_j,&x\in[c_j,c_j+3]\cup[-c_j^\alpha,-c_j^\alpha+1],\quad j\geq1,\\
 2,&\text{otherwise}.
 \end{cases}
\end{equation}
The intervals in \eqref{eq:fixed-dyadic-selector} are pairwise disjoint, and $c(x)$ is defined using all $j\geq1$ so that it is independent of $N$ and of the exponents $p,q$.

We first treat the regime where at least one of $p$ or $q$ is finite by placing a finite train of bumps into one of the two inputs according to the size of $\alpha$, and we then address the remaining infinite-exponent cases.

\medskip
\noindent\textbf{Case 1: $\alpha>1$ (with $1 \leq p < \infty$, $1 \leq q \leq \infty$).}
Choose nonnegative functions $f\in C_c^\infty((-2,2))$ and
$\psi\in C_c^\infty((2,3))$ such that $f=1$ on $[-1,1]$ and
$\psi \not \equiv 0$. For each $N \geq 1$, define
$$
 g_N(x) = \sum_{j=1}^N \psi(x - c_j^\alpha - c_j).
$$
For $K$ sufficiently large, the bumps defining $g_N$ have pairwise
disjoint supports, so for all $1 \leq q \leq \infty$ (with the convention $N^{1/\infty} = 1$),
$$
 \Vert f\Vert_{L^p(\R)} \Vert g_N \Vert_{L^q(\R)}
  = N^{1/q} \Vert f\Vert_{L^p(\R)} \Vert\psi\Vert_{L^q(\R)}.
$$
Put $I_j=[c_j,c_j+1]$. We show that $BSH_{c(x)}^\alpha(f,g_N)(x)$ is bounded below by a positive constant independent of $N$ on each $I_j$ for $1\leq j\leq N$.

Fix $x = c_j + y \in I_j$, where $0 \leq y \leq 1$. Since $\supp f \subset (-2,2)$, any nonzero contribution to the integral requires $x - t \in (-2,2)$, and thus
$$
 c_j - 2 < t < c_j + 3.
$$
In particular, $t > 0$ and $c_j/2 < t < 2c_j$ for $K$ large. Moreover, since both endpoints of $[c_j+(c_j-2)^\alpha,c_j+1+(c_j+3)^\alpha]$ divided by $c_j+c_j^\alpha$ tend to $1$ as $c_j \to \infty$, choosing $K$ sufficiently large ensures that
$$
 \frac{c_j+c_j^\alpha}{2}<x+t^\alpha<2(c_j+c_j^\alpha).
$$
Because the scales satisfy $c_{k+1} + c_{k+1}^\alpha \geq 16(c_k+c_k^\alpha)$ for all $k \geq 1$, the support of the $k$th bump of $g_N$ lies outside this interval whenever $k \neq j$. Thus only the $j$th bump of $g_N$ contributes on $I_j$.

Make the change of variables
$$
 z = x + t^\alpha - (c_j^\alpha + c_j) = y + t^\alpha - c_j^\alpha.
$$
The support condition $\supp \psi \subset (2,3)$ enforces $2 < z < 3$, and hence $1 < z - y < 3$. By the mean value theorem, there exists $\theta \in (c_j, t)$ such that $z - y = t^\alpha - c_j^\alpha = \alpha \theta^{\alpha-1}(t - c_j)$. In particular, after increasing $K$ if necessary,
$$
 0 < t - c_j = \frac{z-y}{\alpha \theta^{\alpha-1}}
 \leq\frac{3}{\alpha c_j^{\alpha-1}} \leq 1,
$$
so the integrand is supported strictly to the right of the singularity $t = c_j$, and
$$
x - t = y - (t - c_j) \in [-1,1],
$$
where $f(x-t)=1$. Since $dt = \frac{1}{\alpha t^{\alpha-1}}\,dz$ and $\theta \geq c_j > t/2$, we have
$$
 \frac{dt}{t-c_j}
 =\left(\frac{\theta}{t}\right)^{\alpha-1}\frac{dz}{z-y}
 \geq\frac{2^{1-\alpha}}{3}\,dz.
$$
Consequently,
\begin{equation}\label{eq:superlinear-output-lower}
 BSH_{c(x)}^\alpha(f,g_N)(x)
 \geq\frac{2^{1-\alpha}}{3}\int_2^3\psi(z)\,dz
 =:a_\alpha>0,\qquad x \in I_j.
\end{equation}
Since the intervals $I_j$ ($1 \leq j \leq N$) are pairwise disjoint and have length $1$, taking $\lambda=a_\alpha/2$ in the definition of the weak $L^r$ quasi-norm gives, for $r<\infty$,
$$
\Vert BSH_{c(x)}^\alpha(f,g_N)\Vert_{L^{r,\infty}(\R)} \geq \frac{a_\alpha}{2} \left|\bigcup_{j=1}^N I_j\right|^{1/r} = \frac{a_\alpha}{2}N^{1/r}.
$$
Therefore,
$$
 \frac{\Vert BSH_{c(x)}^\alpha(f,g_N)\Vert_{L^{r,\infty}(\R)}}
      {\Vert f\Vert_{L^p(\R)}\Vert g_N\Vert_{L^q(\R)}}
 \gtrsim_{\alpha,p,q} N^{1/r-1/q} = N^{1/p} \to \infty
 \qquad \text{as } N \to \infty,
$$
which establishes the failure for $\alpha > 1$ whenever $1 \leq p < \infty$ and $1 \leq q \leq \infty$.

\medskip
\noindent\textbf{Case 2: $\alpha=1$ (with $1 \leq p < \infty$, $1 \leq q \leq \infty$).}
The same argument as in Case 1 applies after choosing a nonnegative $f\in C_c^\infty((-4,2))$ with $f=1$ on $[-3,1]$. For $x = c_j + y \in I_j = [c_j, c_j + 1]$, the support of $f$ restricts $t$ to $c_j - 2 < t < c_j + 5$, which again isolates the $j$th bump of $g_N(x) = \sum_{k=1}^N \psi(x - 2c_k)$. Setting $z = x + t - 2c_j \in (2,3)$, we have
$$
t - c_j = z - y \in (1,3) \qquad \text{and} \qquad x - t = 2y - z \in (-3,0) \subset [-3,1],
$$
so $f(x-t) = 1$ and
$$
BSH_{c(x)}^1(f, g_N)(x) = \int_2^3 \frac{\psi(z)}{z-y}\,dz \geq \frac{1}{3}\int_2^3 \psi(z)\,dz > 0
$$
on $I_j$. The rest of the estimate is identical to Case 1.

\medskip
\noindent\textbf{Case 3: $0 < \alpha < 1$ (with $1 \leq p \leq \infty$, $1 \leq q < \infty$).}
Here we keep the second input fixed and insert the bump train into the first input. Choose a nonnegative function
$g\in C_c^\infty((-1,3))$ with $g=1$ on $[0,2]$, and with the same $\psi\in C_c^\infty((2,3))$ as in Case 1 define
$$
 f_N(x)=\sum_{j=1}^N\psi(-x-c_j^\alpha-c_j).
$$
The bump supports are pairwise disjoint, so for all $1 \leq p \leq \infty$,
$$
 \Vert f_N \Vert_{L^p(\R)} \Vert g\Vert_{L^q(\R)}
 = N^{1/p} \Vert \psi \Vert_{L^p(\R)} \Vert g \Vert_{L^q(\R)}.
$$
Put $I_j=[-c_j^\alpha,-c_j^\alpha+1]$ and fix $x=-c_j^\alpha+y\in I_j$, where $0 \leq y \leq 1$ and $1 \leq j \leq N$. Since $0 < \alpha < 1$ and $(1+v/c_j^\alpha)^{1/\alpha}\to 1$ uniformly for $v\in[-2,3]$, we may choose $K$ sufficiently large that
$$
c_j^\alpha \leq \frac{c_j}{4}, \qquad 4\alpha c_j^{\alpha-1}\leq1,\qquad \frac{c_j}{2} \leq(c_j^\alpha+v)^{1/\alpha}\leq2c_j \quad(-2\leq v\leq3)
$$
for every $j\geq1$.

First, we check that negative $t$ values do not contribute on $I_j$. Under the odd power convention, $t < 0$ gives $x+t^\alpha < x \leq -c_j^\alpha + 1 < -1$, so $g(x+t^\alpha)=0$. Under the even power convention, $g(x+|t|^\alpha) \neq 0$ requires $|t|^\alpha \in (c_j^\alpha - 2, c_j^\alpha + 3)$, and hence $|t| \geq c_j/2$; for $t < 0$ this gives $x - t \geq c_j/2 - c_j^\alpha > 0$, whereas $\supp f_N \subset (-\infty, 0)$. Thus negative $t$ contributions vanish under either convention.

For $t>0$, any nonzero contribution of $g(x+t^\alpha)$ requires
\begin{equation}\label{eq:g_t_restriction}
    t^\alpha\in(c_j^\alpha-2,c_j^\alpha+3),
 \qquad \text{and hence} \qquad \frac{c_j}{2}\leq t\leq2c_j.
\end{equation}
On the other hand, the $k$th bump of $f_N(x-t)$ is $\psi(t+c_j^\alpha-y-c_k^\alpha-c_k)$, which is nonzero only when
$$
t = c_k + c_k^\alpha - c_j^\alpha + y + z \qquad \text{for some } z \in (2,3).
$$
If $k < j$, then since $c_k^\alpha \leq c_k$, $y + z < 4$, and $c_k = 2^{-K(j-k)}c_j \leq 2^{-4}c_j = c_j/16$, we have
\begin{equation}\label{eq:t_upper_bd}
    t \leq 2c_k + 4 \leq \frac{c_j}{8}+4 < \frac{c_j}{2}.
\end{equation}
If $k > j$, then since $y + z > 2$ and $c_k \geq 16c_j$, we have
\begin{equation}\label{eq:t_lower_bd}
    t > c_k - c_j^\alpha \geq 16c_j - \frac{c_j}{4} > 2c_j.
\end{equation}
Both \eqref{eq:t_upper_bd} and \eqref{eq:t_lower_bd} contradict \eqref{eq:g_t_restriction}, so only the $j$th bump can contribute on $I_j$.

For the $j$th bump, put $s = t - c_j$. On its support, $f_N(x-t) = \psi(s-y)$ with $s - y \in (2,3)$, so $2 < s < 4$. By concavity of $u \mapsto u^\alpha$ and our choice of $K$,
$$
 0 \leq (c_j+s)^\alpha-c_j^\alpha
 \leq \alpha c_j^{\alpha-1}s \leq 4\alpha c_j^{\alpha-1} \leq 1.
$$
Thus $x + t^\alpha = y + (c_j + s)^\alpha - c_j^\alpha \in [0,2]$, where $g \equiv 1$. The integrand is supported strictly to the right of $t = c_j$, and substituting $z = s - y$ gives
$$
 BSH_{c(x)}^\alpha(f_N,g)(x) =\int_2^4\psi(s-y)\,\frac{ds}{s} =\int_2^3\frac{\psi(z)}{y+z}\,dz \geq \frac14\int_2^3\psi(z)\,dz =: a_\alpha > 0, \qquad x \in I_j.
$$
Since the intervals $I_j$ ($1 \leq j \leq N$) are disjoint and have length $1$, taking $\lambda = a_\alpha/2$ in the weak $L^r$ quasi-norm yields
$$
\Vert BSH_{c(x)}^\alpha(f_N,g)\Vert_{L^{r,\infty}(\R)} \geq \frac{a_\alpha}{2}N^{1/r},
$$
and hence
$$
 \frac{\Vert BSH_{c(x)}^\alpha(f_N,g)\Vert_{L^{r,\infty}(\R)}}
      {\Vert f_N\Vert_{L^p(\R)}\Vert g\Vert_{L^q(\R)}}
 \gtrsim_{\alpha,p,q}N^{1/r-1/p}
 =N^{1/q} \to \infty
 \qquad \text{as } N \to \infty
$$
whenever $1 \leq q < \infty$ and $1 \leq p \leq \infty$.

\medskip
\noindent\textbf{Case 4: The remaining infinite-exponent cases.}
Notice that Cases 1--2 already cover $q = \infty$, $1 \leq p < \infty$ when $\alpha \geq 1$, and Case 3 already covers $p = \infty$, $1 \leq q < \infty$ when $0 < \alpha < 1$. The remaining endpoint combinations are handled by replacing the bump train in the $L^\infty$ slot with a single smooth flat--top cutoff whose support grows with $N$ while its $L^\infty(\R)$ norm remains $1$:

\begin{itemize}
\item \textit{Subcase 4a: $q=\infty$ and $1 \leq p < \infty$ (for any $\alpha > 0$).} Fix a nonzero nonnegative $f\in C_c^\infty((0,1))$ and choose $g_N \in C_c^\infty(\mathbb R)$ satisfying $0\leq g_N\leq1$ and
$$
g_N = 1 \qquad \text{on } \left[0,\, c_N+3+(c_N+3)^\alpha\right].
$$
For $x = c_j + y$ with $2 \leq y \leq 3$ ($1 \leq j \leq N$), we have $c(x) = c_j$, and the substitution $u = x - t \in (0,1)$ gives $t = c_j + y - u > c_j$ and $x + t^\alpha \in [0, c_N + 3 + (c_N+3)^\alpha]$, so $g_N(x + t^\alpha) = 1$. Hence
$$
 BSH_{c(x)}^\alpha(f,g_N)(x)
 =\int_0^1\frac{f(u)}{y-u}\,du
 \geq\frac13\int_0^1f(u)\,du>0, \qquad x \in [c_j+2, c_j+3].
$$
Since there are $N$ disjoint unit intervals $[c_j+2, c_j+3]$ and $\Vert g_N\Vert_{L^\infty(\R)} = 1$, the weak $L^p$ norm of $BSH_{c(x)}^\alpha(f,g_N)$ grows at least as $N^{1/p} \to \infty$.

\item \textit{Subcase 4b: $p = \infty$, $1 \leq q < \infty$, and $\alpha \geq 1$.} Take $g = \psi \in C_c^\infty((2,3))$ and choose $f_N\in C_c^\infty(\mathbb R)$ satisfying $0\leq f_N\leq1$ and
$$
f_N=1 \qquad \text{on } [-c_N-c_N^\alpha-3,0].
$$
For $x = -c_j^\alpha + y \in I_j=[-c_j^\alpha,-c_j^\alpha+1]$ ($1 \leq j \leq N$), we have $c(x) = c_j$. For $t > 0$, $\psi(x + t^\alpha) \neq 0$ forces $z := x + t^\alpha \in (2,3)$, so
$$
t = (c_j^\alpha + z - y)^{1/\alpha} \in (c_j, c_j + 1)
$$
and
$$
x - t \in [-c_N - c_N^\alpha - 3, 0],
$$
where $f_N(x-t) = 1$. By the same mean-value calculation as in Case 1, the contribution from $t>0$ is at least
$$
a_\alpha = \frac{2^{1-\alpha}}{3}\int_2^3 \psi(z)\,dz > 0.
$$
The contribution from $t<0$ vanishes under the odd power convention; under the even power convention, $x + |t|^\alpha = z \in (2,3)$ forces $t = -(c_j^\alpha + z - y)^{1/\alpha} < -c_j$, so $|t - c_j| > 2c_j$ and the $t < 0$ integral has absolute value at most
$$
\frac{1}{2\alpha c_j^\alpha}\int_2^3\psi(z)\,dz \leq \frac{a_\alpha}{2}
$$
for $K$ sufficiently large. Therefore $BSH_{c(x)}^\alpha(f_N,\psi)(x) \geq a_\alpha/2$ on $I_j$ for $1 \leq j \leq N$, and hence
$$
\Vert BSH_{c(x)}^\alpha(f_N,\psi)\Vert_{L^{q,\infty}(\R)}
\geq \frac{a_\alpha}{4}N^{1/q} \to \infty \qquad \text{as } N \to \infty,
$$
while $\Vert f_N\Vert_{L^\infty(\R)} = 1$ and $\Vert\psi\Vert_{L^q(\R)}$ is fixed.

\item \textit{Subcase 4c: $p=q=r=\infty$ (for any $\alpha > 0$).} Recall from \eqref{eq:fixed-dyadic-selector} that $c(x)=2$ near the origin. For $0<\varepsilon<1/8$, choose $f_\varepsilon\in C_c^\infty((-3,-2-\varepsilon))$ satisfying $0\leq f_\varepsilon\leq1$ and
$$
f_\varepsilon=1
\qquad \text{on } \left[-\frac{5}{2},-2-2\varepsilon\right],
$$
and choose a fixed $g\in C_c^\infty(\mathbb R)$ satisfying $0\leq g\leq1$ and $g=1$ on $[0,1+4^\alpha]$. For $0\leq x\leq\varepsilon/2$, $f_\varepsilon(x-t) \neq 0$ forces $t \in (x + 2 + \varepsilon, x + 3) \subset (2, 4)$, so $x + t^\alpha \in [0, 1 + 4^\alpha]$ where $g = 1$, all contributions to the integral are positive, and
$$
BSH_{c(x)}^\alpha(f_\varepsilon,g)(x) \geq
\int_{x+2+2\varepsilon}^{x+5/2}\frac{dt}{t-2} = \log\left(\frac{x+1/2}{x+2\varepsilon}\right) \geq \log\left(\frac{1}{5\varepsilon}\right).
$$
Thus $\Vert BSH_{c(x)}^\alpha(f_\varepsilon, g) \Vert_{L^\infty(\R)} \to \infty$ as $\varepsilon \to 0^+$ while $\Vert f_\varepsilon \Vert_{L^\infty(\R)} = \Vert g \Vert_{L^\infty(\R)} = 1$.
\end{itemize}
This completes the proof of Proposition \ref{prop:dyadic}. \qed

\section{Open Questions}

While we are able to completely determine the Lebesgue-space boundedness of $BM_a^\alpha$ and $BSH_c^\alpha$, for $BH_a^\alpha$ our counterexample in Theorem \ref{thm:main1} only rules out strong--type boundedness at the $L^1$ endpoint. We therefore propose the following conjecture.

\begin{conjecture}\label{open_problem}
For $\alpha \in (0,\infty) \setminus \{1\}$ and $a:\R \to \R$ measurable (or smooth and globally Lipschitz with $0 < A < a(x) < B < \infty$), $BH_a^\alpha$ maps boundedly from $L^p(\R) \times L^q(\R)$ into $L^r(\R)$ for $1 < p,q,r < \infty$ with $\frac{1}{p} + \frac{1}{q} = \frac{1}{r}$, and into $L^{1,\infty}(\R)$ when $r = 1$.
\end{conjecture}

Proving such a statement is beyond the elementary positive-domination methods that were used to resolve Theorem \ref{thm:main2}, and instead one will most likely be required to utilize the \textsf{LGC methodology}, first introduced by Lie in \cite{UnifiedLie} and further refined in many subsequent works; see for instance \cite{BHL25}, \cite{GaitanLie1}, \cite{GaitanLie2}, \cite{HL24}, \cite{HL26}, and \cite{HL23}.

Our result for $BSH_c^\alpha$ establishes the existence of measurable functions $c$ which cause the operator to fail all strong and weak--type estimates. However, this leaves open the question of whether there are regularity and nondegeneracy conditions one can impose on $c$ that allow boundedness to be recovered. Assuming $c \in C^1(\R)$ so that discussing the derivative is meaningful, we formulate the following conjecture.

\begin{conjecture}\label{open_problem2}
    For $\alpha \in (0,\infty) \setminus \{1\}$ and $c \in C^1(\R)$, suppose there are $\epsilon_1,\epsilon_2 > 0$ such that for all $x \in \R$:
    \begin{enumerate}
        \item\label{i:1} $|c'(x) - 1| > \epsilon_1$;
        \item\label{i:3} $|1 + \alpha |c(x)|^{\alpha-1}c'(x)| > \epsilon_2$.
    \end{enumerate}
    Then for
    $$
    \frac{1}{2} < r < \infty, \qquad 1 < p,q < \infty, \qquad \frac{1}{r} = \frac{1}{p} + \frac{1}{q},
    $$
    there is a constant $C = C(c,\alpha,p,q)$ such that
    $$
    \Vert BSH_c^\alpha(f,g) \Vert_{L^r(\R)} \leq C\Vert f \Vert_{L^p(\R)} \Vert g \Vert_{L^q(\R)}. 
    $$
\end{conjecture}

We notice that conditions \eqref{i:1} and \eqref{i:3} rule out the example constructed in the proof of Proposition \ref{prop:unit-interval-shift}. This also explains why we do not use the much simpler choice $c_x = x$ in that proof, as we wish to demonstrate that $c'(x) - 1$ getting asymptotically close to $0$ without vanishing identically is already sufficient to destroy strong and weak--type bounds. Establishing Conjecture \ref{open_problem2} will again likely be based on the LGC methodology discussed.

One can of course consider the ``flat" variants of Conjectures \ref{open_problem} and \ref{open_problem2} by setting $\alpha = 1$. The conjectures are still expected to hold, but the proofs of such statements one would expect to be much more technically challenging and based on wave packet analysis. Thus we propose attacking the curved variants first, namely $\gamma(t) = t^3$, $\gamma(t)=t^2$ and then $\gamma(t)=t$, in increasing order of their expected difficulty.

\section{AI Declaration}\label{sec:AI_dec}

The first author devoted part of his PhD studies to seeking $L^2 \times L^2 \to L^1$ bounds for the first two operators considered in this work. With recent advances in AI, he began exploring these questions through discussions with Gemini 3.1 Pro. The model suggested that the bounds might fail, although its proposed arguments were incomplete. This suggestion motivated the authors to investigate unboundedness. The resulting counterexamples and the methods underlying their construction were developed by the authors and subsequently extended to the third operator. The construction of a counterexample with positive, bounded coefficients was also discussed with Codex. GPT--5.6 Luna assisted in drafting the text which briefly summarizes the result in \cite{SteinConjecture} concerning the Stein conjecture. All mathematical arguments and proofs were developed and verified by the authors, who reviewed and revised the AI-assisted text and take full responsibility for the manuscript.

\end{document}